\pdfoutput=1 
\pdfoutput=1 
\pdfoutput=1 
\pdfoutput=1 
\pdfoutput=1 
\documentclass[10pt, a4paper, oneside, reqno]{amsart}

\makeatletter
\def\@settitle{\begin{center}%
		\baselineskip14\p@\relax
		\normalfont\LARGE\scshape\bfseries
		\@title
	\end{center}%
}
\makeatother

\makeatletter

\def\subsection{\@startsection{subsection}{2}%
	\z@{.5\linespacing\@plus.7\linespacing}{.5\linespacing}%
	{\normalfont\large\bfseries}}

\def\subsubsection{\@startsection{subsubsection}{3}%
	\z@{.5\linespacing\@plus.7\linespacing}{.5\linespacing}%
	{\normalfont\itshape}}

\usepackage[usenames, dvipsnames]{color}
\definecolor{darkblue}{rgb}{0.0, 0.0, 0.45}

\usepackage[colorlinks	= true,
raiselinks	= true,
linkcolor	= darkblue, %MidnightBlue,
citecolor	= Mahogany,
urlcolor	= ForestGreen,
pdfauthor	= {Peyman Mohajerin Esfahani},
pdftitle	= {},
pdfkeywords	= {},
pdfsubject	= {},
plainpages	= false]{hyperref}

\usepackage{dsfont,amssymb,amsmath,subcaption, graphicx,enumitem} %,etaremune,savesym}
\usepackage{amsfonts,dsfont,mathtools, mathrsfs,amsthm} 
\usepackage[amssymb, thickqspace]{SIunits}
\usepackage{algorithm,algorithmicx,fancyhdr}
\usepackage{tabularx}
\usepackage{makecell}

\allowdisplaybreaks
\date{\today}
\theoremstyle{theorem}
\newtheorem{Thm}{Theorem}[section]

\newtheorem{Lem}[Thm]{Lemma}
\newtheorem{Cor}[Thm]{Corollary}
\newtheorem{As}[Thm]{Assumption}

\newtheorem{Def}[Thm]{Definition}

\newtheorem{Rem}[Thm]{Remark}

\theoremstyle{remark}

\definecolor{green}{rgb}{0,.6,0}

\newcommand{\R}{\mathbb{R}}

\newcommand{\N}{\mathbb{N}}

\newcommand{\Let}{\coloneqq}

\newcommand{\RNum}[1]{\uppercase\expandafter{\romannumeral #1\relax}}

\newcommand{\im}{\mbox{Im}}
\newcommand{\basis}{F_{\mathrm b}}

\graphicspath{{./fig/}}

\title[From Static to Dynamic Anomaly Detection]
{From Static to Dynamic Anomaly Detection 
\\with Application to Power System Cyber Security}

\author{Kaikai~Pan,
	Peter~Palensky,
	and~Peyman~Mohajerin~Esfahani}%
	\thanks{The authors are with the Delft University of Technology, The Netherlands (email: \tt{\{K.Pan, P.Palensky,  P.MohajerinEsfahani\}}@tudelft.nl).}

\begin{document} 
\maketitle

\begin{abstract}
Developing advanced diagnosis tools to detect cyber attacks is the key to security of power systems. It has been shown that multivariate data injection attacks can bypass bad data detection schemes typically built on static behavior of the systems, which misleads operators to disruptive decisions. In this article, we depart from the existing static viewpoint to develop a diagnosis filter that captures the dynamics signatures of such a multivariate intrusion. To this end, we introduce a dynamic residual generator approach formulated as robust optimization programs in order to detect a class of disruptive multivariate attacks that potentially remain stealthy in view of a static bad data detector. We investigate two possible desired features: (i) a non-zero transient and (ii) a non-zero steady-state behavior of the residual generator in the presence of an attack. In case (i), the problem is reformulated as a finite, but possibly non-convex, optimization program. We further develop a linear programming relaxation that improves the scalability, and as such practicality, of the diagnosis filter design. In case (ii), it turns out that the resulting robust program admits an exact convex reformulation, yielding a Nash equilibrium between the attacker and the residual generator. This assertion has an interesting implication: the proposed approach is not conservative in the sense that the additional knowledge of the worst-case attack does not improve the diagnosis performance. 
To illustrate our theoretical results, we implement the proposed diagnosis filter to detect multivariate attacks on the system measurements deployed to generate the so-called Automatic Generation Control signals in a three-area IEEE 39-bus system. %
\end{abstract}

%===============================================================================
\section{Introduction} 
\label{sec:intro}
%===============================================================================

The digital transformation of our power system does not only lead to better observability, flexibility and efficiency, but also introduces a phenomenon that is new to power system controls: cyber security threats. NIST \cite{NIST2018} defines five functions for protecting Information and Communication Technology (ICT): (i) Identify, (ii) Protect, (iii) Detect, (iv) Respond, (v) Recover. It would be naive to think an ICT system can be perfectly protected  in order to address the issues raised by (iii)-(v). This paper focuses on (iii) Detection for supervisory control and data acquisition (SCADA) systems, which are in charge of transmitting measurement and control signals between power system substations and control centers. %~\cite{Kirschen2009}. 
Such SCADA systems are notorious for being based on legacy ICT, and are a popular target for adversaries \cite{Gorman2009, ChenAbu-Nimeh2011} nowadays. The consequences of a successful attack on SCADA systems can be catastrophic to an economy and society in general \cite{mohajerin2010cyber,Liang2017a}. In this light, it is of utmost importance to detect these attacks and respond accordingly. Notably, if the malicious attacks can be detected sufficiently fast, the corrupted signals can be disconnected or corrected by resilient controls, preventing further severe damage \cite{Tiniou2013}.%

\paragraph*{\bf Literature on anomaly detection}
Traditionally, SCADA systems deploy bad data detection~(BDD) to filter out possible erroneous measurements due to sensor failures or anomalies~\cite{Teixeira2010}. %AburExposito2004, 
The BDD process captures only a snapshot of the steady states of system trajectories, and thus only exploits possible {\em static} impact of intrusions. Although this method can perform successfully in detecting basic attacks, it may fail in the presence of the so-called {\em stealthy multivariate attacks} that carefully launch synthesized false data injections given full knowledge of the system model~\cite{Hug2012}.%

It was first explored in \cite{LiuNingReiter2009} that such an attack can perturb the state estimation function without triggering alarms in BDD. Since then vulnerability and impact analysis of stealthy attacks on power systems have been a prominent subject in the literature. A typical notion to quantify the vulnerability to stealthy attacks is directly concerned with the level of efforts required to alter specific measurements \cite{Giani2013, Pan2018}. Without advanced diagnosis tools, tampering measurements remains undetected, causing state deviations, equipment damages or even cascading failures~\cite{Liang2016}. Techniques proposed to deal with stealthy attacks include statistical methods such as sequential detection using Cumulative Sum (CUSUM)-type algorithms~\cite{Li2015}, and measurements consistency assessment under certain observability assumptions~\cite{Zhao2018}. A detection method that leverages online information is described in \cite{Ashok2018}, which is applicable by ensuring the availability and accuracy of load forecasts and generation schedules. In \cite{Liu2014a}, a mechanism is introduced to formulate the detection scheme as a matrix separation problem, but it only recovers intrusions among corrupted measurements over a particular period of time.% 

These techniques are essentially static detection methods that may be confined by certain prior assumptions on the distribution of measurement errors. Despite an extensive and ongoing literature focusing on the static part of BDD mechanism, the following question remains largely unexplored:
\begin{flushleft}
	\centering
	{\em Would it be possible to detect stealthy multivariate attacks in a real-time operation by exploiting the attack impact on the dynamics of system trajectories during the transient?}
\end{flushleft} 
The importance of an appropriate answer to this question has been reinforced thanks to recent advances in sensing technology in the modern power systems. Our main objective in this article is to address this question.% 

\paragraph*{\bf Related work} 
Detection methods concerning system dynamics have primarily emerged under the topic of {\em fault detection and isolation filters}. A subclass of these schemes is the observer-based approach applied initially to linear models~\cite{Massoumnia1989}; see also \cite{Ding2008} for a comprehensive summary of the large body of literature. The authors in~\cite{Nyberg2006} further extend the modeling framework to general linear differential-algebraic equations (DAEs), enhancing the applicability of such methods particularly for power system applications due to the common governing physical laws in this setting. Recently, a variant of observer-based methods is also investigated in~\cite{Ameli2018a} so as to deal with unknown natural exogenous inputs.%

An inherent shortcoming of many observer-based approaches is that the degree of the resulting diagnosis filter is effectively the same as the system dynamics, which may yield an unnecessarily complex filter in large-scale power systems. To our best of knowledge, there are relatively much fewer studies in the literature on the design of the reduced-order observers where the conditions for a minimum order existence need to be satisfied \cite{Ding2008, Gao2017}. The closest approach in the literature is~\cite{Esfahani2016} where a scalable optimization-based filter design is developed for high-dimensional nonlinear control systems. However, the proposed method opts for mainly dealing with a single fault scenario, and may not be as effective in case of smart multivariate adversarial inputs. 

An effective approach toward security and modeling the interaction between attackers and detectors builds on the rich framework of game theory. Recently, the authors in~\cite{Shukla2019} propose a two player mixed strategy game to address a dynamic resource-planning problem between an attacker targeting the communication equipment and a defender protecting the control network. Similar frameworks have also been deployed to model the dynamics of information flow between an advanced persistent threat and a detector \cite{Sahabandu2018, Sahabandu2019}.

\paragraph*{\bf Our contributions} 
The main objective of this article is to develop a {\em diagnosis filter} to detect {\em multivariate data injection attacks} in a real-time operation. For this purpose, considering a class of disruptive multivariate attack scenarios (Definition~\ref{def:disrupt_attack}), we first characterize the attack impact on power system dynamics through a set of differential equations. Having transferred the dynamics into the discrete-time domain, we further restrict the diagnosis filter to a family of dynamic residual generators that entirely decouples the contributions of the attacks from the system states and natural disturbances. In order to identify an admissible multivariate attack scenario, we propose an optimization-based framework to robustify the diagnosis filter with respect to such attacks, i.e., aiming to design a filter whose residual (output) is sensitive to any plausible disruptive multivariate attacks. The main contributions of this article are as follows:
\begin{itemize}
	\item[(i)] Unlike the existing literature, we go beyond a static viewpoint of anomaly detection to capture the attack impact on the dynamics of system trajectories. To this end, we characterize the diagnosis filter design approach as a robust optimization program. It is guaranteed that while the filter residual is decoupled from system states and disturbances, it still remains sensitive to all admissible disruptive multivariate attacks even if the attacker has full knowledge about the diagnosis filter architecture (Definition~\ref{def:robust_detect} and the program~\eqref{opt:max-min-opt}). % 
	
	\item[(ii)] To detect attacks during the transient behavior, we reformulate the resulting robust program as a finite, possibly non-convex, optimization program~(Theorem~\ref{the:max-min-reform}). To improve the scalability of the proposed solution, we further propose a linear programming relaxation which is highly tractable for large scale systems~(Corollary~\ref{cor:convex_rel}). It is guaranteed that if the optimal value of the relaxed program is positive, the resulting diagnosis filter is able to detect any admissible disruptive attack scenarios, which may remain stealthy through the lens of a static detector.%
	
	\item[(iii)] We further explore the steady-state behavior of the diagnosis filter in the presence of a plausible attack scenario (Lemma~\ref{lem:robust_scheme_nz}). In this case, we develop an exact convex reformulation of the resulting robust program. As a byproduct, we show that the proposed solution is indeed a Nash equilibrium (saddle point) between the attacker and the residual generator (Theorem~\ref{the:max-min-reform-nz-ne}). An interesting implication of such a Nash equilibrium is that the information of the attack signal may not necessarily improve the performance of the diagnosis filter. In other words, if the proposed convex optimization fails to have a desirable feasible solution, it then implies that there exists a disruptive stealthy attack where the exact knowledge of the attack signal still does not help design a successful residual generator.%
\end{itemize}

In addition to the above theoretical results, we validate the performance and effectiveness of the proposed diagnosis filter on a multi-area IEEE 39-bus system. Numerical results illustrate that the diagnosis filter successfully generates a residual ``alert'' in the presence of multivariate attacks that are stealthy in a static viewpoint, even in a noisy environment with imprecise measurements. %

Section~\ref{sec:prob} introduces the problem of power system cyber security, and the challenges posed by multivariate attacks are highlighted. Section~\ref{sec:powersys_model} discusses a model instance of power system dynamics under attacks on measurements. Our diagnosis filter design is proposed in Section~\ref{sec:detect} where an optimization framework is introduced, and numerical simulations are reported in Section~\ref{sec:results}. %

\paragraph*{\bf Notation}
The symbols $\mathbb{R}$, $\mathbb{N}$ represent the set of real numbers and integers, respectively. Given a matrix $A \in \mathbb{R}^{m \times n}$, $A^\top$ denotes its transpose, and the space $\mbox{Im}(A)$ represents its range space. Throughout the paper, the matrix $I$ is the identity matrix with an appropriate dimension. Given a column vector~$a \in \mathbb{R}^{m}$, $\mbox{diag}(a)$ denotes an $m \times m$ diagonal matrix with the elements of vector $a$ sitting on the main diagonal and the rest of the elements being zero. We also denote by $\mbox{diag}[A_1, \ A_2, \ \dots, \ A_k]$ a block matrix whose main diagonal elements are the matrices $A_1, \, A_2, \, \dots, \, A_k$. Given a vector~$a \in \mathbb R^m$, the associated $\ell_{\infty}-$norm is denoted by $\| a \|_{\infty} = \max_{i\le m}|a_{i}|$. %

%=============================================================================== 
\section{Problem Statement}\label{sec:prob}

\subsection{Static detection and system modeling} \label{subsec:basic}

\begin{figure}[t!p]
	\centering
	\includegraphics[scale=1.25]{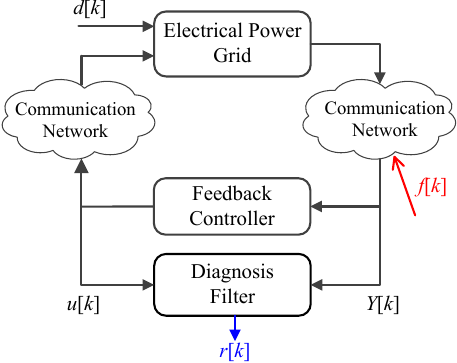}
	\caption{Schematic block diagram of the system model.}
	\label{fig:sys_schem}
\end{figure}

For a power grid, measurements are collected by remote sensors and transmitted through a SCADA network. The typical BDD is conducted to detect the erroneous measurements at each time instance. We can see this as a static process: it only concerns the system states $X[k] \in \mathbb{R}^{n_{X}}$ and measurements $Y[k] \in \mathbb{R}^{n_{Y}}$ at time step $k \in \mathbb{N}$, which can be described by 
\begin{equation}\label{eq:static_sys}
Y[k] =   CX[k] + D_{f}f[k],
\end{equation} 
where $C \in \mathbb{R}^{n_{Y}\times n_{X}}$ is the measurement matrix, and $f[\cdot] \in \mathbb{R}^{n_{f}}$ represents the data injection attacks on measurements. Note that the matrix $D_{f}$ characterizes which measurement is vulnerable to attacks. It is customary to define a {\em residual signal} for a static detector,  $r_{S}[k] := Y[k] - \hat{Y}[k]$, where $\hat{Y}[\cdot]$ denotes the estimated measurements. In the traditional weighted least squares estimation, the estimate of state is $({C}^\top{C})^{-1}{C}^\top Y[k]$, assuming that $C$ has full column rank with high measurement redundancy. Then the measurements estimate is $C({C}^\top{C})^{-1}{C}^\top Y[k]$, and the residual signal %of BDD 
%can be further expressed as
can be further expressed as
\begin{equation}\label{eq:static_r}
r_{S}[k] = \big(I - C({C}^\top{C})^{-1}{C}^\top \big)Y[k].
\end{equation} 

Such an anomaly detector has shown a good effectiveness in detecting erroneous data and basic attacks \cite{Deng2018}. However, in the face of coordinated attacks on multiple measurements, this static detector can fail. In this article, motivated by this shortcoming, we take a dynamic design perspective where we shift the emphasis on an attack as a static process to its effects on power system dynamics. In particular, we opt for differentiating the attack impact on the systems trajectories from natural disturbances such as load deviations. %

To model its impact on the dynamics, let us consider a more general modeling framework in Figure~\ref{fig:sys_schem}. The electrical grid is operated by a digital controller that receives measurements as inputs and sends control signals to the actuators through communication networks. These transmitted data are applied in discrete-time samples. On the power grid side, the input $d[k] \in \mathbb{R}^{n_{d}}$ represents natural disturbances. On the controller side, a control signal $u[k] \in \mathbb{R}^{n_{u}}$ is computed given the measurements $Y[k]$. Note that with the closed-loop control, the corruptions $f[k]$ on the measurements would affect the system dynamics. The dynamics of the closed-loop system is
\begin{equation}\label{eq:dis_sys}
\left\{
\begin{aligned}
& X[k+1] = A_{x}X[k] + B_{d}d[k] + B_{u}u[k], \\
& Y[k] =   CX[k] + D_{f}f[k], 
\end{aligned}
\right.
\end{equation}
where $A_{x}$, $B_{d}$ and $B_{u}$ are constant matrices. Let us highlight the difference between the dynamical system~\eqref{eq:dis_sys} and the respective static counterpart~\eqref{eq:static_sys}. In fact, the time independence of the first equation in~\eqref{eq:dis_sys} describes the dynamics of the system, while the algebraic equation~\eqref{eq:static_sys} represents the relation on each time instance and describes a static relation between the states and outputs. The aim of this study is to exploit such dynamics information in~\eqref{eq:dis_sys} in order to design a diagnosis filter to detect stealthy multivariate attacks. To illustrate the attack impact on the system dynamics, we can simply consider the feedback controller as a linear operator such that $u[k] = GY[k]$ where $G\in \R^{n_{u}\times n_Y}$ is a matrix gain. By defining the closed-loop system matrices $A_{cl}:= A+B_{u}GC$ and $B_{f}:=B_{u}GD_{f}$, we can reformulate \eqref{eq:dis_sys} into
\begin{equation}\label{eq:dis_cls}
\left\{
\begin{aligned}
& X[k+1] = A_{cl}X[k] + B_{d}d[k] + B_{f}f[k], \\
& Y[k] =   CX[k] + D_{f}f[k]. \\
\end{aligned}
\right.
\end{equation}
\begin{Rem}[Dynamic feedback controller]\label{rem:dyn_feedback_ctrl}
	The restriction to only a static feedback controller~$u[k]=GY[k]$ to transfer from \eqref{eq:dis_sys} to \eqref{eq:dis_cls} is without loss of generality. Namely, the proposed framework is rich enough to subsume a dynamic controller architecture as well. Indeed, when the controller has certain dynamics, it suffices to augment the system dynamics~\eqref{eq:dis_sys} with the controller states and outputs. We refer to Appendix~\ref{subsec:app_rem_dyn_ctrl}, for such a detailed analysis. \par
\end{Rem}
\begin{Rem}[Attacks impact on the dynamics of system trajectories]\label{rem:att_affect_dyn}
	In light of \eqref{eq:dis_cls}, matrices $B_{f}$, $D_{f}$ capture the attack impact on the power system dynamics, mapping attacks $f[\cdot]$ to the system states and measurements respectively. \par
\end{Rem}

In the following, we show that the state-space description \eqref{eq:dis_cls} is a particular case of DAE model. By introducing a time-shift operator $q$ : $qX[k] \rightarrow X[k+1]$, one can fit \eqref{eq:dis_cls} into 
\begin{equation}\label{eq:dae}
H(q)x[k] + L(q)y[k] + F(q)f[k] = 0,
\end{equation}
where $x:= [X^\top \ d^\top]^{\top}$ represents the unknown signals of system states and disturbances; $y := Y$ contains all the available data for the operator. Let $n_{x}$ and $n_{y}$ be the dimensions of $x[\cdot]$, $y[\cdot]$. We denote $n_{r}$ as the number of rows in \eqref{eq:dae}. Then $H, \ L, \ F$ are polynomial matrices in terms of the time-shift operator $q$ with $n_{r}$ rows and $n_{x}, n_{y}, n_{f}$ columns separately, by defining,
\begin{equation}\label{eq:dae_def}
H(q) := \left[\begin{matrix} -qI + A_{cl} & B_{d}\\ C & 0\\ \end{matrix}\right], \quad L(q) := \left[\begin{matrix} 0 \\ -I \\ \end{matrix}\right], \quad F(q) := \left[\begin{matrix} B_{f}\\ D_{f} \\ \end{matrix}\right]. \nonumber 
\end{equation}

\subsection{Challenge: multivariate attacks}\label{subsec:challenge}

We start this subsection with an existing result characterizing the set of stealthy multivariate attacks that can bypass the static detector. %
\begin{Lem}[{Stealthy attack values~\cite[Theorem 1]{LiuNingReiter2009}}]\label{lem:stealth set}
	Consider the measurement equation~\eqref{eq:static_sys} and the static detector with the respective residual function~\eqref{eq:static_r}. Then, an attack $f[\cdot]$ remains stealthy, i.e., it does not cause any additional residue to~\eqref{eq:static_r}, if it takes values from the set
	\begin{equation}\label{eq:att_sp}
	\mathcal{F} \Let \big\{ f[k] \in \mathbb{R}^{n_{f}}: \ D_{f}f[k] \in \im (C), \quad k \in \N \big\},
	\end{equation}
\end{Lem}

One can observe that a stealthy attack~$D_{f}f[\cdot]$ described in \eqref{eq:att_sp} has the knowledge of the system model~\eqref{eq:static_sys} through the range space of $C$. That is, it represents a tampered value~$D_{f}f[k] = C \Delta X$ where $\Delta X \in \mathbb{R}^{n_{X}}$ can be any injected bias influencing certain sensor measurements. Such multivariate attacks would also challenge the detector design as they may neutralize the diagnosis filter outputs.

\begin{As}[Stationary attacks]\label{ass:station_att}
	Throughout this article, we consider attacks~$f[\cdot]$ that are time-invariant, i.e., $f[k] = 0$ for all $k \le k_{\min}$; $f[k]= f \in \mathcal{F}$ for all $k > k_{\min}$. Namely, the attack occurs as a constant bias injection~$f$ on measurements during the system operations at a specific unknown time instance $k_{\min}$, and it remains unchanged since then. 
\end{As}

Advanced attacks also pursue a maximized impact on the system dynamics. Thus, an adversary would try to inject ``{\em smart}" false data, possibly with large magnitudes, in such a way that it causes the maximum damage. The next definition opts to formalize this class of attacks.
\begin{Def}[Disruptive stealthy attack] \label{def:disrupt_attack}
	Consider a set of vectors~$\basis \Let [f_{1}, f_{2}, \dots, f_{d}]$ representing a finite basis for the set of stealthy attacks~\eqref{eq:att_sp}, i.e., the set~$\mathcal{F}$ defined in \eqref{eq:att_sp} can equivalently be represented by
	\begin{align*}
	\mathcal{F} = \left\{\basis^\top\alpha =  \sum_{i=1}^{d}\alpha_i f_i ~\Big |~  \alpha = [\alpha_{1}, \alpha_{2}, \cdots, \alpha_{d}]^\top \in \R^d \right \} \,.
	\end{align*}
	We call a signal~$f\in\mathcal F$ {\em disruptive stealthy attack} if its corresponding coefficients~$\alpha$ is a polytopic set, i.e., it belongs to
	\begin{equation}\label{eq:alpha_set}
	\mathcal{A} : = \big\{ \alpha \in \mathbb{R}^{d} \ | \ A\alpha \geq b  \big\},
	\end{equation}
	where $A \in \R^{n_{b}\times d}$ and $b \in \R^{n_{b}}$ are given matrices. We emphasize that the subsequent analysis and the proposed diagnosis filter design only rely on the convexity of the set~$\mathcal{A}$. Namely, the choice~\eqref{eq:alpha_set} may be adjusted according to the application at hand, as long as the convexity of the set is respected. 
\end{Def}
%

%=============================================================================== 
\section{Cyber Security of Power Systems: AGC modeling} \label{sec:powersys_model}

In this section, we first go through a modeling instance of power system dynamics in the form of \eqref{eq:dis_cls}: Automatic Generation Control (AGC) closed-loop system under attacks. This model will be used to validate our diagnosis filter. Figure~\ref{fig:39bus} depicts the diagram of a three-area IEEE 39-bus system. AGC is a feedback controller that tunes the setpoints of participated generators (e.g., G11 of Area 1) to maintain the frequency as its nominal value and the tie-line (e.g., L1-2 between Area 1 and 2) power as the scheduled one. %

\begin{figure}[t!p]
	\centering
	\includegraphics[scale=0.56]{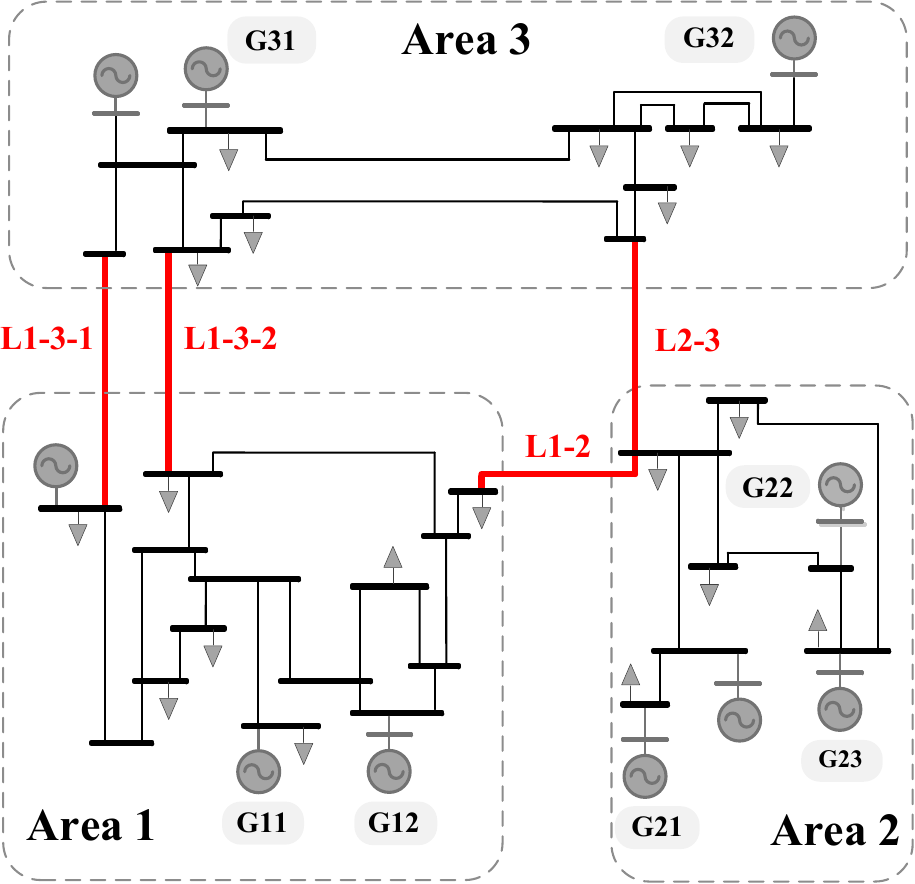}
	\caption{Three-area 39-bus system: the measurements of the tie-lines (in red) L1-3, L1-2, L2-3 are attacked.}
	\label{fig:39bus}
\end{figure}

In the work of AGC, a linearized model is commonly used for the load-generation dynamics \cite{Rakhshani2017}. For a three-area system, the frequency dynamics in Area $i$ can be written as
\begin{subequations}\label{eq:dynamics}
	\begin{align}\label{eq:freq}
	\Delta \dot{\omega}_{i} = \frac{1}{2H_{i}}(\Delta P_{m_{i}} - \Delta P_{tie_{i}}  -\Delta P_{l_{i}} - D_{i} \Delta \omega_{i}), 
	\end{align}
	where $H_{i}$ is the equivalent inertia constant; $D_{i}$ is the damping coefficient and $\Delta P_{l_{i}}$ denotes load deviations. Here $\Delta P_{tie_{i}}$, $\Delta P_{m_{i}}$ represent the total tie-line power exchanges from Area $i$ and the total generated power in Area $i$, i.e., $\Delta P_{tie_{i}} = \sum_{j \in \mathcal{E}_{i}} \Delta P_{tie_{i,j}}$ where $\mathcal{E}_{i}$ denotes the set of areas that connect to Area $i$, and $\Delta {P}_{m_{i}} = \sum_{g=1}^{G_{i}} \Delta {P}_{m_{i,g}}$ where $G_{i}$ denotes the number of participated generators in Area $i$, and we have %
	\begin{align} %\label{eq:tiepower}
	\Delta \dot{P}_{m_{i,g}} &= -\frac{1}{T_{ch_{i,g}}}(\Delta P_{m_{i,g}} + \frac{1}{S_{i,g}} \Delta \omega_{i} - \phi_{i,g}\Delta P_{agc_{i}}), \\
	\Delta \dot{P}_{tie_{i,j}} &=  T_{ij} (\Delta \omega_{i} - \Delta \omega_{j}),
	\end{align}
	where $T_{ch_{i,g}}$ is the governor-turbine's time constant; $S_{i,g}$ denote the droop coefficient; $T_{ij}$ is the synchronizing parameter between Area $i$ and $j$. Note that $\Delta P_{agc_{i}}$ is the signal from AGC for the participated generators to track the load changes, and $\phi_{i,g}$ is the participating factor, i.e., $\sum_{g=1}^{G_{i}} \phi_{i,g} =1$.
	After receiving the frequency and tie-line power measurements, the {\em area control error} (ACE) is computed for an integral action,
	\begin{align} 
	ACE_{i} &= {B}_{i} \Delta \omega_{i} + \sum_{j \in \mathcal{E}_{i}} \Delta P_{tie_{i,j}}, \\
	\Delta \dot{P}_{agc_{i}} &=  - K_{I_{i}} ACE_{i}, \label{eq:ace_agc} %
	\end{align}
\end{subequations}
where $B_{i}$ is the frequency bias and $K_{I_{i}}$ represents the integral gain. %
Based on the equations \eqref{eq:dynamics}, the linearized model of Area $i$ can be presented as the state equation 
\begin{equation}\label{eq:spx_areai}
\dot{X}_{i}(t) = A_{ii} X(t) + B_{i,d}d_{i}(t) + \sum_{j \in \mathcal{E}_{i}} A_{ij} X_{j}(t), \\
\end{equation}
where $X_{i}$ is the state vector; $d_{i} := \Delta P_{l_{i}}$ denotes load deviations. Recall Remark~\ref{rem:dyn_feedback_ctrl} that \eqref{eq:spx_areai} is an augmented model for the closed-loop AGC system that $X_{i}$ consists of not only the electrical grid states (e.g., frequency, generator output and tie-line power) but also the controller state $\Delta {P}_{agc_{i}}$, i.e.,
\begin{equation}\label{eq:spx_xi}
X_{i} := \left[\begin{matrix} \{\Delta P_{tie_{i,j}}\}_{j \in \mathcal{E}_{i}} & \Delta \omega_{i} & \{\Delta {P}_{m_{i,g}}\}_{1:G_{i}} & \Delta P_{agc_{i}} \end{matrix}\right]^{\top}. \nonumber
\end{equation}
Besides in \eqref{eq:spx_areai}, $A_{ii}$ is the system matrix of Area $i$; $A_{ij}$ is a matrix whose only non-zero element is $-T_{ij}$ in row 1 or 2 and column 3;  $B_{i,d}$ is the matrix for load deviations. 

In addition to $\eqref{eq:spx_areai}$, we assume a measurement model with high redundancy that the measurements of each tie-line power ($\Delta P_{tie_{i,j}}$) and the total tie-lines' power ($\Delta P_{tie_{i}}$), the frequency ($\Delta \omega_{i}$), each generator output ($\Delta P_{m_{i,g}}$) and the total generated power ($\Delta P_{m_{i}}$), and the AGC controller output ($\Delta {P}_{agc_{i}}$) are all available. Besides, vulnerabilities within SCADA networks may allow cyber intrusions. Thus the output equation is
\begin{equation}\label{eq:spy_areai_f}
Y_{i}(t) = C_{i}X(t) +  D_{i,f}f_{i}(t),
\end{equation}
where $Y_{i}$ is the system output and $C_{i}$ is the output tall-matrix with full column rank. Here $f_{i}$ denotes multivariate attacks and the matrix $D_{i,f}$ quantifies which output is attacked. In the aforementioned section, due to the feedback loop, attacks on the measurements would also affect the frequency dynamics. Hence the state equation \eqref{eq:spx_areai} during attacks becomes
\begin{equation}\label{eq:spx_areai_f}
\dot{X}_{i}(t) = A_{ii} X(t) + B_{i,d}d_{i}(t) + B_{i,f}f_{i}(t) + \sum_{j \in \mathcal{E}_{i}} A_{ij} X_{j}(t), \nonumber\\
\end{equation}
where $B_{i,f}$ is the matrix that relates attacks to system states. 

Using the state equations of each area, the continuous-time model of the three-area %
system can be obtained,
\begin{equation}\label{eq:spx_39}
\dot{X}(t) = \tilde{A}_{cl} X(t) + \tilde{B}_{d}d(t) + \tilde{B}_{f}f(t), \\
\end{equation}
where $X$ is the vector consisting of groups of dynamic states in each area; $d$ is the vector for all areas' load deviations; $f$ denotes all the attack signals in the three-area, namely,
\begin{gather}\label{eq:spx_39_x&d&f}
X= \left[\begin{matrix} X_{1}^{\top} & X_{2}^{\top} & X_{3}^{\top} \end{matrix}\right]^{\top}, \quad d = \left[\begin{matrix} \Delta P_{l_{1}} & \Delta P_{l_{2}} & \Delta P_{l_{3}}\end{matrix}\right]^\top, \quad f= \left[\begin{matrix} f_{l}^\top & f_{2}^\top & f_{3}^\top \end{matrix}\right]^\top. \nonumber
\end{gather} 
In \eqref{eq:spx_39}, $\tilde{A}_{cl}$ is the closed-loop system matrix; $\tilde{B}_{d}$, $\tilde{B}_{f}$ are constant matrices that relate load deviations and attacks to system states. For the three-area system, these matrices are
\begin{gather}\label{eq:spx_ABdBf}
\tilde{A}_{cl} = \left[\begin{matrix} A_{11} & A_{12} & A_{13} \\ A_{21} & A_{22} & A_{23} \\  A_{31} & A_{32} & A_{33} \end{matrix}\right], \quad 
\tilde{B}_{d} = \mbox{diag} \left[\begin{matrix}B_{1,d}, \ B_{2,d}, \ B_{3,d} \end{matrix}\right]\, , \quad \tilde{B}_{f} = \mbox{diag} \left[\begin{matrix}B_{1,f}, \ B_{2,f}, \ B_{3,f} \end{matrix}\right] \, . \nonumber
\end{gather}
We can also obtain the output equation of the system,
\begin{equation}\label{eq:spy_39}
Y(t) = CX(t) + D_{f}f(t),
\end{equation}
where $Y$ is the system output vector containing all the three areas' outputs; $C$ is the output matrix; $D_{f}$ quantifies all the vulnerable signals. Similarly, these matrices are
\begin{gather}\label{eq:spx_39_y&c&df}
Y= \left[\begin{matrix} Y_{1}^\top & Y_{2}^\top & Y_{3}^\top \end{matrix}\right]^\top, \quad C = \mbox{diag} \left[\begin{matrix}C_{1}, \ C_{2}, \ C_{3} \end{matrix}\right], \quad D_{f} = \mbox{diag} \left[\begin{matrix}D_{1,f}, \ D_{2,f}, \ D_{3,f} \end{matrix}\right]. \nonumber
\end{gather}

To obtain the sampled discrete-time model as \eqref{eq:dis_cls}, \eqref{eq:spx_39} and \eqref{eq:spy_39} must be discretized. We deploy a zero-order hold (ZOH)\footnote{The inputs signals $d(\cdot)$ and $f(\cdot)$ in \eqref{eq:spx_39} are assumed to be piecewise constant within the sampling periods.} discretization for a given sampling period $T_{s}$ \cite{Ogata1995}, 
\begin{equation}\label{eq:zoh}
A_{cl} = e^{\tilde{A}_{cl}T_{s}}, \quad B_{d} = \int_{0}^{T_s} e^{\tilde{A}_{cl}(T_{s}-t)}\tilde{B}_{d} \mathrm{d} t. \\
\end{equation}
Note that the attack matrix $\tilde{B}_{f}$ has the same matrix transformation as $\tilde{B}_{d}$, resulting $B_{f}$. The above approximation is exact for a ZOH and \eqref{eq:zoh} corresponds to the analytical solution of the discretization. Therefore, the above model can be described in the form of \eqref{eq:dis_cls} which again can be fitted into the DAE \eqref{eq:dae}. In Appendix~\ref{subsec:agc_39_sys_matrices}, we provide the detailed description of the involved parameters of the three-area 39-bus system as well as the attack scenarios on the AGC measurements. %

%=============================================================================== 
\section{Robust Dynamic Detection} \label{sec:detect}

\subsection{Preliminaries for diagnosis filter construction}\label{subsec:pre_att_detect}
An ideal detection aims to implement a non-zero mapping from the attack to the diagnostic signal while decoupled from system states and disturbances, given the available data $y[\cdot]$ in the control center. In the power system dynamics described via a set of DAE, we restrict the diagnosis filter to a type of dynamic residual generator in the form of linear transfer functions, i.e., $r_{D}[k] := R(q)y[k]$ where $r_{D}$ is the residual signal of the diagnosis filter and $R(q)$ is a transfer operator. Note that $y[\cdot]$ is associated with the polynomial matrix $L(q)$ in \eqref{eq:dae}. We propose a formulation of transform operator $R(q)$ as
\begin{equation}\label{filter}
R(q) := a(q)^{-1}N(q)L(q), \nonumber
\end{equation}
where $N(q)$ is a polynomial vector with the dimension of $n_{r}$ and a predefined order $d_{N}$. To make $R(q)$ physically realizable, stable dynamics $a(q)$ with sufficient order need to be added as the denominator where all the roots are strictly contained in the unit circle. Note that, unlike the observer-based methods, here $d_{N}$ can be much less than the dimension of system dynamics. Then $N(q)$ and $a(q)$ are the two variables for a diagnosis filter design. By multiplying $a(q)^{-1}N(q)$ in the left of \eqref{eq:dae}, we have
\begin{equation}\label{eq:residual_gen}
\begin{aligned}
r_{D}[k] &= a(q)^{-1}N(q)L(q)y[k] = -\underbrace{a(q)^{-1}N(q)H(q)x[k]}_{\text{(\RNum{1}})} - \underbrace{a(q)^{-1}N(q)F(q)f[k]}_{\text{(\RNum{2}})},
\end{aligned}
\end{equation}
where term $\text{(\RNum{1}})$ in \eqref{eq:residual_gen} is due to $x[\cdot]$ of system states and natural disturbances. Term $\text{(\RNum{2}})$ is the desired contribution from the attacks $f[\cdot]$. In view of this diagnosis filter description, we introduce a class of residual generator which is sensitive to disruptive stealthy attacks as defined in Definition~\ref{def:disrupt_attack}. 

\begin{Def}[Robust residual generator]\label{def:robust_detect}
	Consider a linear residual generator represented via a polynomial vector~$N(q)$. This residual generator is robust with respect to disruptive stealthy attacks introduced in~Definition~\ref{def:disrupt_attack} if
	\begin{align}
	\label{eq:robust-poly}
	\left\{
	\begin{array}{ll}
	(I) &  N(q)H(q) = 0, \\ 
	(II) & N(q)F(q)\basis \alpha \neq 0, \quad \forall \alpha \in \mathcal{A}, 
	\end{array}
	\right.
	\end{align}
	where the basis matrix~$\basis$ and the set $\mathcal A$ are the same as the ones in Definition~\ref{def:disrupt_attack}. 
\end{Def}

In the next step, we show that the polynomial equations~\eqref{eq:robust-poly} in Definition~\ref{def:robust_detect} can be characterized as a feasibility problem of a finite robust program. 

\begin{Lem}[Linear program characterization]\label{lem:robust_scheme_reform}		
	Consider the polynomial matrices $H(q) = \sum_{i = 0}^{1}H_{i}q^{i}$, $N(q) := \sum_{i = 0}^{d_{N}}N_{i}q^{i}$ and $F(q) = F$, where $H_{i} \in \mathbb{R}^{n_{r} \times n_{x}}$, $N_{i} \in \mathbb{R}^{n_{r}}$, and $F \in \R^{n_{r} \times n_{f}}$ are constant matrices. Then, the family of robust residual generators in~\eqref{eq:robust-poly} is characterized by% the set of algebraic relations
	\begin{align}\label{eq:robust-poly_bar}
	\left\{
	\begin{array}{ll} 
	(I) & \bar{N}\bar{H} = 0, \\ 
	(II)& \big\| \bar{N}V(\alpha) \big \|_{\infty} > 0, \quad \forall \alpha \in \mathcal{A},
	\end{array}
	\right. 
	\end{align}
	where $\|\cdot\|_{\infty}$ denotes the infinite vector norm, and
	\begin{align*}
	\bar{N} \Let\left[\begin{matrix} N_{0} & N_{1} &  \cdots & N_{d_{N}} \end{matrix}\right], \quad \bar{H} \Let \left[\begin{matrix} H_{0} & H_{1} & 0 & \cdots & 0 \\ 0 & H_{0} & H_{1} & 0 & \vdots \\ \vdots & 0  & \ddots & \ddots & 0 \\ 0 & \cdots & 0 & H_{0} & H_{1} \end{matrix}\right], \quad 
	V(\alpha) \Let \left[\begin{matrix} F\basis\alpha & 0 & \cdots & 0 \\ 0 & F\basis\alpha & 0 & \vdots \\ \vdots & 0 & \ddots & 0 \\ 0 & \cdots & 0 & F\basis\alpha \end{matrix}\right]. \nonumber
	\end{align*}
%	\begin{equation}\label{eq:NHF_bar}
%	\begin{aligned}
%	\bar{N} &\Let\left[\begin{matrix} N_{0} & N_{1} &  \cdots & N_{d_{N}} \end{matrix}\right], \nonumber\\
%	\bar{H} &\Let \left[\begin{matrix} H_{0} & H_{1} & 0 & \cdots & 0 \\ 0 & H_{0} & H_{1} & 0 & \vdots \\ \vdots & 0  & \ddots & \ddots & 0 \\ 0 & \cdots & 0 & H_{0} & H_{1} \end{matrix}\right], \nonumber\\
%	V(\alpha) &\Let \left[\begin{matrix} F\basis\alpha & 0 & \cdots & 0 \\ 0 & F\basis\alpha & 0 & \vdots \\ \vdots & 0 & \ddots & 0 \\ 0 & \cdots & 0 & F\basis\alpha \end{matrix}\right]. \nonumber
%	\end{aligned}
%	\end{equation}
	% 
\end{Lem}
\begin{proof}
	The proof follows a similar line of arguments as \cite[Lemma~4.2]{Esfahani2016}.
	The key step %in the proof 
	is to observe that $N(q)H(q) = \bar{N}\bar{H} [I ,\ qI ,\ \cdots ,\ q^{d_N+1}I ]^\top$, and $N(q)F\basis\alpha = \bar{N}V(\alpha) [I ,\ qI ,\ \cdots ,\ q^{d_N}I ]^\top$. The rest of the proof follows rather 
	straightforwardly, and we omit the details for %the sake of 
	brevity. 
\end{proof}

\subsection{Robust diagnosis filter: transient behavior} \label{subsec:maxmin_detect}

In light of \eqref{eq:robust-poly_bar}, we can define a symmetric set for the design variable $\bar{N}$ of the dynamic residual generator,
\begin{equation}\label{eq:sets_barN}
\mathcal{N} : = \{\bar{N} \in \mathbb{R}^{(d_{N}+1)n_{r}} \ | \ \bar{N}\bar{H} = 0, \  \|\bar{N}\|_{\infty} \leq \eta \}. \\
\end{equation}
The second constraint in the set is added to avoid possible unbounded solutions. To design a robust residual generator, we aim to find an $\bar{N} \in \mathcal{N}$ that for all  $\alpha \in \mathcal{A}$, \eqref{eq:robust-poly_bar} can be satisfied. To this end, a natural reformulation of the residual synthesis is to consider an objective function as the second quantity in~\eqref{eq:robust-poly_bar}  influenced by the parameters~$\mathcal{N}$ and the attacker action~$\alpha$, i.e., $\mathcal{J}(\bar{N},\alpha) := \| \bar{N}V(\alpha) \|_{\infty}$. A successful scenario from an attacker viewpoint is to minimize this objective function given a residual generator. Therefore, we take a rather conservative viewpoint where the attacker may have complete knowledge of the system model and even the residual generator parameters, and exploits it so as to synthesize a stealthy attack. We then reformulate the diagnosis filter 
design as the robust optimization program,
\begin{align}\label{opt:max-min-opt}
\gamma^\star \Let \max\limits_{ \bar{N} \in \mathcal{N} } \ \min\limits_{\alpha \in \mathcal{A} } \  \Big\{ \mathcal{J}(\bar{N},\alpha):= \| \bar{N}V(\alpha) \|_{\infty} \Big\}.
\end{align}

The optimal value~$\gamma^\star$ of the robust reformulation~\eqref{opt:max-min-opt} is indeed an indication whether the attack still remains stealthy in the dynamic setting, i.e., if $\gamma^\star > 0$ then the optimal solution~$\bar{N}^\star$ yields a diagnosis filter in the form of~\eqref{filter} which detects all the admissible attacks introduced in Definition~\ref{def:disrupt_attack}. However, if $\gamma^\star = 0$, then it implies that for any possible detectors (static or dynamic) there exists a stationary disruptive attack that remains stealthy. In the next step, we show that the robust program~\eqref{opt:max-min-opt} can be equivalently reformulated as a finite (non-convex) optimization problem.

\begin{Thm}[Finite reformulation of \eqref{opt:max-min-opt}]\label{the:max-min-reform} 
	The robust optimization~\eqref{opt:max-min-opt} can be equivalently described via the finite optimization program
	\begin{align}\label{opt:max-min-ref}
	&\begin{array}{ccl}
	\gamma^\star = & \max\limits_{\bar{N}, \ \beta, \ \lambda} &  b^{\top}\lambda \\
	& \mbox{s.t.} & \sum\limits_{i=0}^{d_{N}}(\beta_{2i}-\beta_{2i+1})N_{i}F\basis = \lambda^{\top}A, \\
	& & \mathbf{1}^{\top}\beta = 1, \ \beta \geq 0, \\
	& & \bar{N} \in \mathcal{N}, \ \lambda \geq 0, 
	\end{array}
	\end{align}
	where $\beta = [\beta_{0}, \ \beta_{1}, \ \cdots, \ \beta_{2d_{N}+1}]^\top$ is an $ \R^{2d_{N}+2}$-valued auxiliary variable.
\end{Thm}

\begin{proof}
	See Appendix~\ref{subsec:app_the_proof1}.
\end{proof}

The exact reformulation program~\eqref{opt:max-min-ref} for \eqref{opt:max-min-opt} is unfortunately non-convex due to the bilinearity between the variables $\beta$ and $N_{i}$ in the first constraint. In the following corollary, we suggest a convex relaxation of the program by restricting the feasible set of the variable~$\beta$ to a $2d_N+2$ finite possibilities where $\beta = [0, \ \cdots, \ 1, \ \cdots, \ 0]^\top$ in which the only non-zero element of the vector is the $i$-${\rm th}$ element.
\begin{Cor}[Linear program relaxation]\label{cor:convex_rel}
	Given $i \in \{1, \ \dots, \ 2d_{N}+2\}$, consider the linear program
	\begin{align}\label{opt:max-min-relax}
	&\begin{array}{ccl}
	\gamma^{\star}_{i} := & \max\limits_{\bar{N}, \, \lambda} &  b^{\top}\lambda \\
	& \mbox{s.t.} & (-1)^{i}N_{\lfloor i/2 \rfloor}F\basis = \lambda^{\top}A, \tag{${\rm LP}_{i}$}\\
	& & \bar{N} \in \mathcal{N}, \ \lambda \geq 0,
	\end{array}
	\end{align}
	where ${\lfloor \cdot \rfloor}$ is the ceiling function that maps the argument to the least integer. Then, the solution to the program~\eqref{opt:max-min-relax} is a feasible solution to the exact robust design reformulation~\eqref{opt:max-min-ref}, and $\max_{\{i\le 2d_N+2\}} \gamma^{\star}_{i} \le \gamma^\star$. In particular, if for any $i \in \{1,\ \dots, \ 2d_N+2\}$ we have $\gamma^{\star}_{i} > 0$, then the solution to \ref{opt:max-min-relax} offers a robust residual generator %diagnosis filter 
	detecting all admissible disruptive attacks introduced by~Definition~\ref{def:disrupt_attack}. 
\end{Cor}

Corollary~\ref{cor:convex_rel} suggests that the maximum optimal value of $\{\gamma^{\star}_{0}, \ \gamma^{\star}_{1}, \ \cdots, \ \gamma^{\star}_{2d_{N}+2}\}$ and its corresponding~$\bar{N}^{\star}$ provide a suboptimal solution to the original robust design~\eqref{opt:max-min-opt}.

%\revised{It is noteworthy that, in this paper for the first time we develop dynamic residual generator with a robust design for stationary multivariate attacks (Assumption~\ref{ass:station_att}). %that can keep stealthy in a static detector. We hope this can be a stepping stone addressing detection of multivariate attacks in power systems. The diagnosis filter would naturally encounter more difficulty if it comes to detect certain smart ``dynamic'' attacks. We have the following remark,

We note that the focus of this article is on stationary (time-invariant) attacks. It is also important to highlight that the robust design perspective~\eqref{opt:max-min-opt} allows the attacker to know the system model and filter parameters. In such a setting, the detection procedure could be much more difficult if the attacker would be able to dynamically adapt the attacks over the time, i.e., the attack signal is time-varying. In fact, in a multivariate attack scenario, one can construct a disruptive time-varying attack bypassing any linear residual generators. The next remark alludes more to this situation.
\begin{Rem}[Time-varying stealthy attacks]\label{rem:time-variant_att}
	Consider a multivariate attack $f=[f_1 \ f_2 \ \cdots \ f_{n_f}]^\top $ where each element is a time-varying signal $f_i = f_i[k]$. Then, the residual~\eqref{eq:residual_gen} can be rewritten as 
	\begin{align}\label{eq:dynamic_attack}
	a(q)r_D[k] = - \sum_{i=1}^{n_f} \Big(N(q)F_i f_i[\cdot]\Big)[k],
	%\frac{N(q)F_1}{N(q)F_{n_f}} f_1 -\frac{N(q)F_2}{N(q)F_{n_f}} f_2 - \ \cdots \ - \frac{N(q)F_{n_f-1}}{N(q)F_{n_f}} f_{n_f-1}. 
	\end{align}
	where $F=[F_1 \ F_2 \ \cdots \ F_{n_f}]$ represents the attack dynamics matrix. One can inspect that when the time-varying relation $\sum_{i=1}^{n_f} \big(N(q)F_i f_i[\cdot]\big)[k] = 0$ holds for every $k$, for instance when 
	\[f_{n_f}[k] = - \big(N(q)F_{n_f}\big)^{-1} \sum_{i=1}^{n_f-1} \Big(N(q)F_i f_i[\cdot]\Big)[k],
	\] 
	then the residual outcome~\eqref{eq:dynamic_attack} stays zero for all $k$, and as such, the attack remains undetected. 
\end{Rem}

The proposed robust design in \eqref{opt:max-min-opt} does not necessarily enforce a non-zero steady-state residual of the diagnosis filter under multivariate attacks. Namely, the design perspective of~\eqref{opt:max-min-opt} focuses on detection of attacks during the transient behavior without any requirements on long-term behavior of the residual. Indeed, the residual signal $r_{D}$ may return to zero value after a successful reaction to the attack occurrence. A more stringent perspective is to require a non-zero steady-state behavior under any admissible attack scenario in $\alpha \in \mathcal A$. This extension is addressed in the next subsection.

\subsection{Robust diagnosis filter: steady-state behavior}
In order to design a diagnosis filter with non-zero steady-state residual ``alert'' when a multivariate attack occurs, the robust optimization~\eqref{opt:max-min-opt} can be modified by a more conservative (smaller) objective function~$\mathcal{J}(\bar{N},\alpha) \Let | \bar{N} \bar{F} \alpha|$ where 
\begin{align}\label{F_bar}
\bar{F} \Let\left[\begin{matrix} F \basis & F \basis &  \cdots & F \basis \end{matrix}\right]^{\top}. 
\end{align}
A similar treatment as the preceding subsection can establish a framework for computational purposes. The next lemma follows similar objective as in Lemma~\ref{lem:robust_scheme_reform} with a more demanding requirement of the non-zero long-term residual behavior. 
%We refer to Appendix-C for further details along this direction. In this subsection we consider the steady-state behavior of the diagnosis filter such that it can have non-zero steady-state residual ``alerts'' when multivariate attacks occur. For that purpose, one can modify Lemma~\ref{lem:robust_scheme_reform} as the following.
%The diagnosis filter design in Subsection~\ref{subsec:maxmin_detect} does not enforce a non-zero steady-state residual under multivariate attacks. In fact, %as Figure~\ref{subfig25:rrn} implies, 
%the residual signal $r_{D}$ of the diagnosis filter under stealthy attacks may become zero. 
%In order to design a diagnosis filter with non-zero steady-state residual ``alerts'' when multivariate attacks occur, one can modify Lemma~\ref{lem:robust_scheme_reform} in the following fashion.}
%
\begin{Lem}[Non-zero steady-state residual characterization]\label{lem:robust_scheme_nz}		
	For the polynomial matrices $H(q)$, $N(q)$ and $F(q)$ as defined in Lemma~\ref{lem:robust_scheme_reform}, the family of dynamic residual generators with non-zero steady-state residual under multivariate attacks can be characterized by the algebraic relations
	\begin{align}\label{eq:robust-poly_bar_nz}
	\left\{
	\begin{array}{ll} 
	(I) & \bar{N}\bar{H} = 0, \\ 
	(II)& | \bar{N} \bar{F} \alpha | > 0, \quad \forall \alpha \in \mathcal{A},
	\end{array}
	\right. 
	\end{align}
	where $\bar{F}$ is defined in \eqref{F_bar}, and the matrices $\bar{N}, \bar{H}$ are as defined in Lemma~\ref{lem:robust_scheme_reform}.
\end{Lem}
\begin{proof}
	Recall that $N(q)H(q) = \bar{N}\bar{H} [I ,\ qI ,\ \cdots ,\ q^{d_{N} + 1}I ]^\top$. Thus if $\bar{N}\bar{H}=0$, the diagnosis filter becomes $r_{D}[k] = -a(q)^{-1}N(q)f[k]$. Note the steady-state value of the filter residual under attacks would be $-a(q)^{-1}N(q)F(q)f|_{q=1}$. Thus for the multivariate attack with $\alpha$, the steady-state value of the filter residual is $-a(1)^{-1}N(1)F(1)\basis \alpha$. The proof concludes by noting that $N(1)F(1)\basis \alpha = \bar{N}\bar{F} \alpha$. 
\end{proof}

In a similar fashion, the robust design perspective in \eqref{opt:max-min-opt} can be modified accordingly as
\begin{align}\label{opt:max-min-opt-nz}
\mu^\star \Let \max\limits_{ \bar{N} \in \mathcal{N} } \ \min\limits_{\alpha \in \mathcal{A} } \  \Big\{ \mathcal{J}(\bar{N},\alpha) \Let | \bar{N} \bar{F} \alpha | \Big\} .
\end{align}
Notice the relation between the new objective function with the absolute value and the one in~\eqref{opt:max-min-opt} with the infinity-norm. As it appears in the next result, the new setting is in fact a restricted case of the finite reformulation in Theorem~\ref{the:max-min-reform}. 

\begin{Thm}[Residual long-term behavior: exact convex reformulation and Nash equilibrium]\label{the:max-min-reform-nz-ne}
	Consider the minimax counterpart of the program~\eqref{opt:max-min-opt} as defined 
	\begin{align}\label{opt:min-max-opt-nz}
	\varphi^\star := \min\limits_{\alpha \in \mathcal{A}} \ \max\limits_{ \bar{N} \in \mathcal{N} } \  \Big\{ \mathcal{J}(\bar{N},\alpha) \Let | \bar{N} \bar{F} \alpha | \Big\}.
	\end{align}
	Each of the program~\eqref{opt:max-min-opt-nz} and \eqref{opt:min-max-opt-nz} can be equivalently reformulated through the linear programs
	\begin{subequations}
		\label{opt:thm:lp}
		\begin{align}
		\label{opt:max-min-opt-ref-nz-lp}
		&\begin{array}{ccl} 
		\mu^\star = &\max\limits_{\bar{N}, \ \lambda} & b^{\top} \lambda \\
		& \mbox{s.t.} & \bar{N}\bar{F}= \lambda^{\top}A \\
		&& \bar{N} \in \mathcal{N}, \ \lambda \geq 0 \,,  
		\end{array} \\ 
		\label{opt:min-max-opt-ref-nz-lp}
		&\begin{array}{ccl} 
		\varphi^\star = & \min\limits_{v_{1}, \, v_{2}, \, w, \, \alpha} & \mathbf{1}^{\top}v_{1} + \mathbf{1}^{\top}v_{2} \\
		& \mbox{s.t.} & \bar{H}w + v_{1} - v_{2} = \bar{F}\alpha \\
		& & v_{1} \geq 0, \ v_{2} \geq 0, \\
		& & A \alpha \geq b \,.
		\end{array}
		\end{align}
		Moreover, the value of each of these two programs coincide, i.e., $\mu^\star = \varphi^\star$. 
	\end{subequations}
\end{Thm}

\begin{proof}
	See Appendix~\ref{subsec:app_the_proof_the2}.
\end{proof}

It is worth noting the difference between the robust perspective of \eqref{opt:max-min-opt-nz} versus the minimax program~\eqref{opt:min-max-opt-nz}. While in the design perspective of~\eqref{opt:max-min-opt-nz} the filter is oblivious to the possible attack scenarios, in the perspective of~\eqref{opt:min-max-opt-nz} the filter is aware of the attack signal and opts to detect that particular signal in the presence of natural disturbances. Obviously, the former setting is the one closer to the reality and, in general, the knowledge of the attack signal should help the detection significantly. This observation can indeed be translated through the usual weak inequality of $\mu^\star \le \varphi^\star$. However, Theorem~\ref{the:max-min-reform-nz-ne} indicates that the filter performance, in view of the long-term behavior of the worst-case attack scenario, indeed does not depend on the exact knowledge of the attacker signal and the inequality holds as the equality. We summarize this discussion in the following remark.  
%
%
%of the above is that the attacker is given with the residual generator parameters in \eqref{opt:max-min-opt-nz} to synthesize a stealthy attack, while in \eqref{opt:min-max-opt-nz} the residual generator is given with the exact knowledge of the attack signal but the attacker is oblivious to the diagnosis filter. In general, these two formulations lead to different results, and the inequality $\varphi^\star \geq \mu^\star$ always holds from the minimax theorem \cite{Irle1985}. However, in the following we show that the solution to \eqref{opt:max-min-opt-nz} and \eqref{opt:min-max-opt-nz} is indeed a Nash equilibrium (saddle point) between the attacker and the residual generator.
%
%
%Next, we say that, for \eqref{opt:max-min-opt-nz} and \eqref{opt:min-max-opt-nz}, it admits a Nash equilibrium (saddle point) such that
%%
%\begin{align}\label{eq:nash_eq}
%\max\limits_{ \bar{N} \in \mathcal{N}} \ \min\limits_{\alpha \in \mathcal{A}} \ \mathcal{J}(\bar{N},\alpha)  \ = \min\limits_{\alpha \in \mathcal{A}} \ \max\limits_{\bar{N} \in \mathcal{N} } \ \mathcal{J}(\bar{N},\alpha). 
%\end{align}
%%
%
%    %
%
%
%First we can have convex reformulations of \eqref{opt:max-min-opt-nz} and \eqref{opt:min-max-opt-nz}. Consider the symmetric feasible set of the filter coefficients~$\mathcal{N}$ as defined in \eqref{eq:sets_barN}. The robust optimizations~\eqref{opt:max-min-opt-nz} and \eqref{opt:min-max-opt-nz} can be equivalently described via the following linear programs respectively,
%%

\begin{Rem}[Nash equilibrium interpretation]\label{rem:nash_equilibrium}
	If the linear programs~\eqref{opt:max-min-opt-ref-nz-lp} \eqref{opt:min-max-opt-ref-nz-lp} admit a positive optimal value $\varphi^\star = \mu^\star > 0$, then the resulting filter can detect all the admissible multivariate attacks described by Definition~\ref{def:disrupt_attack} along with a non-zero steady-state residual level. On the other hand, if the optimal values coincide with $\varphi^\star = \mu^\star = 0$, it then implies that there is no linear filter being able to decouple the admissible attack with~$\alpha^\star$, the solution to \eqref{opt:min-max-opt-ref-nz-lp}, from the natural disturbances in a long-term horizon. 
	
	%	If one can find an optimal solution with a positive optimal value $\mu^\star$ by solving \eqref{opt:max-min-opt-ref-nz-lp}, and thus $\varphi^\star = \mu^\star > 0$ as indicated by \eqref{eq:nash_eq} in Theorem~\ref{the:max-min-reform-nz-ne}, the diagnosis filter catches all the multivariate attacks and can have non-zero steady-state residual ``alerts'' even if the attacker has full knowledge of the residual generator parameters. On the other hand, if the linear programs \eqref{opt:max-min-opt-ref-nz-lp} \eqref{opt:min-max-opt-ref-nz-lp} are feasible but $\varphi^\star = \mu^\star = 0$, the additional information of the attack signal does not improve the diagnosis performance %there is no diagnosis filter of such kind that can do a better ``job'' 
	%	(note in \eqref{opt:min-max-opt-nz} the filter is given with the knowledge of the attack signal %$\alpha$ 
	%	). It cannot have non-zero steady-state residual ``alert'' when the multivariate attacks occur, and the exact knowledge of the attack signal does not help the detection process. %whatever the attacker knows the residual generator parameters or not. 
	%	However, even when $\varphi^\star = \mu^\star = 0$, we may still find an $\bar{N}^{\star}$ with $\gamma^\star > 0$ in \eqref{opt:max-min-opt} that the diagnosis filter detects attacks during the transient behavior.
\end{Rem}

%=============================================================================== 
\section{Numerical Results} 
\label{sec:results}
%=============================================================================== 

\subsection{Test system and diagnosis filter description}\label{subsec:sys_descrip}

In order to validate the effectiveness of the diagnosis filter with application to power system cyber security, we employed the IEEE 39-bus system which is well-known as a standard system for testing of new power system analysis. As shown in Figure~\ref{fig:39bus}, this system consists of 3 areas and 10 generators where 7 of them are equipped with AGC for frequency control. All the participating generators in each area are with equal participation factors. The total load of the three-area system is ${5.483}\ \mathrm{GW}$ for the base of ${100}\ \mathrm{MVA}$ and ${60}\ \mathrm{Hz}$. The generator specifications and AGC parameters of each area are referred to \cite{bevrani2008}, and the linear frequency dynamics model has been developed in the preceding Section~\ref{sec:powersys_model}. Thus we result in a 19-order model in the form of \eqref{eq:dis_cls}.

We apply the diagnosis filter proposed in Section~\ref{sec:detect} to detect multivariate disruptive attacks on the measurements of AGC system. In the following simulations, we set the degree of the dynamic residual generator $d_{N} = 3$ which is much less than the order of the dynamics model, the sampling time $T_s = {0.5}\ \mathrm{sec}$ and the finite time horizon ${60}\ \mathrm{sec}$. To design the filter, we set the denominator in the form $a(q) = (q - p)^{d_{N}}/{{(1-p)}^{d_{N}}}$ where $p$ is a user-defined variable acting as the {\em pole} of the transfer operator $R(q)$, and it is normalized %
in steady-state value for all feasible poles. The pole is set to be $p=0.8$ for a stable dynamic behavior at the beginning, and we have deployed the solver CPLEX to solve the corresponding optimization problems.  

\subsection{Simulation results}
\label{subsec:sim_results}

\begin{figure*}[t!p]
	\centering
	\begin{subfigure}[t]{0.49\textwidth}
		\centering
		\includegraphics[scale=0.52]{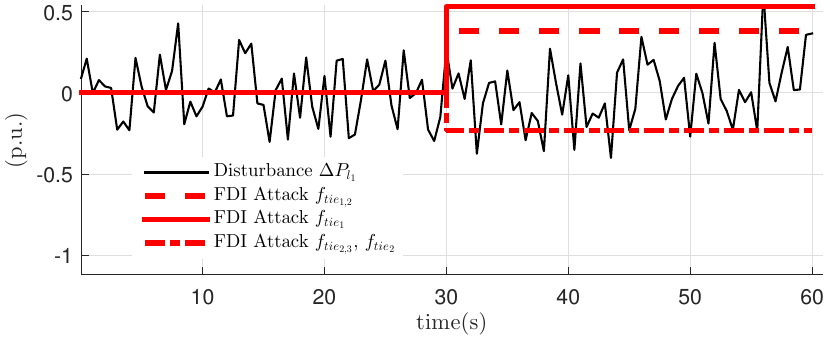}
		\caption{Load disturbance and basic attack}\label{subfig11:ad}
	\end{subfigure}
	~
	\begin{subfigure}[t]{0.49\textwidth}
		\centering
		\includegraphics[scale=0.52]{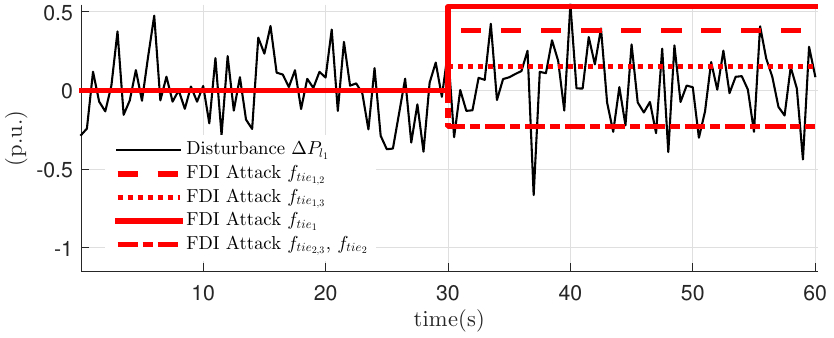}
		\caption{Load disturbance and stealthy attack}\label{subfig21:ad}
	\end{subfigure}
	\\ 
	\begin{subfigure}[t]{0.49\textwidth}
		\centering
		\includegraphics[scale=0.52]{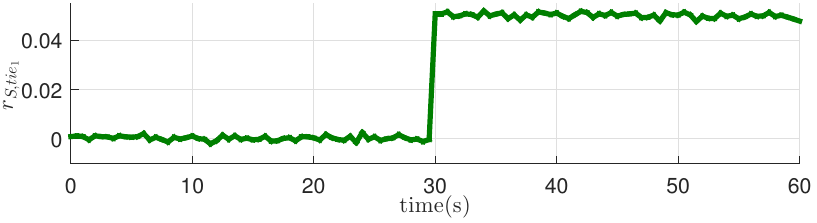}
		\caption{Residual of static detector under basic attack}\label{subfig14:srn}
	\end{subfigure}
	~	
	\begin{subfigure}[t]{0.49\textwidth}
		\centering
		\includegraphics[scale=0.52]{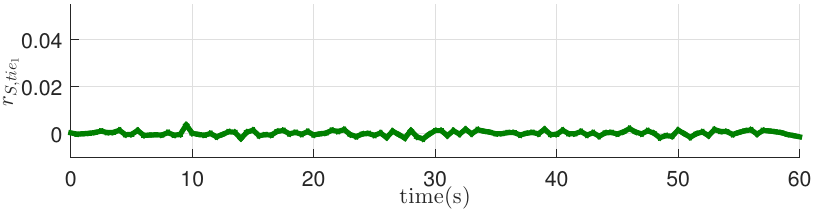}
		\caption{Residual of static detector under stealthy attack}\label{subfig24:srn}
	\end{subfigure}
	\\ 
	\begin{subfigure}[t]{0.49\textwidth}
		\centering
		\includegraphics[scale=0.52]{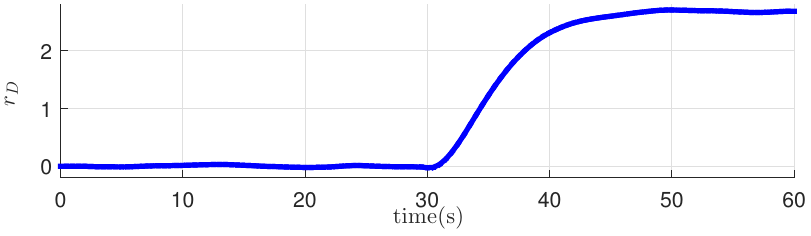}
		\caption{Residual of dynamic detector under basic attack}\label{subfig15:rrn}
	\end{subfigure}
	~
	\begin{subfigure}[t]{0.49\textwidth}
		\centering
		\includegraphics[scale=0.52]{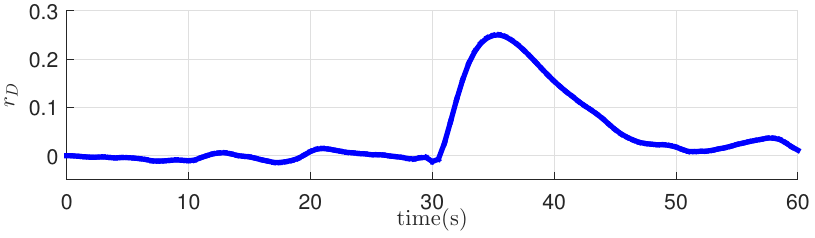}
		\caption{Residual of dynamic detector under stealthy attack}\label{subfig25:rrn}
	\end{subfigure}
	\caption{Static detector in \eqref{eq:static_r} versus dynamic detector (diagnosis filter) from Corollary~\ref{cor:convex_rel} under basic and stealthy attacks.} % Noises were added to the AGC process and measurements by the covariance matrices $R_{X}$ and $R_{Y}$. $r_{S,tie_1}$ corresponds to the residue of the measurement $\Delta P_{tie_{1}}$.} % caption for whole figure
	\label{fig:resfig_(un)stealth}
\end{figure*}

To evaluate the performance of the diagnosis filter, the disturbances $d_{i} = \Delta P_{l_{i}}$ are modeled as stochastic load patterns. To capture its uncertainty, as shown in Figure~\ref{subfig11:ad} and Figure~\ref{subfig21:ad}, we mainly model $\Delta P_{l_{1}}$ in Area 1 as random zero-mean Gaussian signals. It should be noted that tie-line power flow measurements %
are much more vulnerable to cyber attacks, comparing with frequency measurements (e.g., the anomalies in frequency can be easily detected by comparing the corrupted reading with the normal one.) \cite{Chen2018}. Therefore as indicated in Figure~\ref{fig:39bus} we mainly focus on the scenario that there are 5 vulnerable tie-line power measurements, namely $\Delta P_{tie_{1,2}}$, $\Delta P_{tie_{1,3}}$, $\Delta P_{tie_{1}}$, $\Delta P_{tie_{2,3}}$ and $\Delta P_{tie_{2}}$. Recalling Definition~\ref{def:disrupt_attack} for %
stealthy attack basis, thus there exist 3 basis vectors in the spanning set and we model them as follows: $f_{1}= [0.1 \ 0 \ 0.1 \ 0 \ 0]^T$, $f_{2}= [0.1 \ 0.15 \ 0.25 \ 0 \ 0]^T$, $f_{3}= [0 \ 0 \ 0 \ 0.1 \ 0.1]^T$ (all in $\mathrm{p.u.}$). Here each basis vector lies in the range space of the output matrix that the corrupted measurements still align with an actual physical state, bypassing the static detector $r_{S}[\cdot]$. %
Furthermore, without loss of generality we set $A=\mathbf{1}^\top$ and $b=1.5$ in the set $\mathcal{A}$ and $\eta = 10$ in the set $\mathcal{N}$. The design variable $\bar{N}$ of the robust residual generator is first derived by solving \eqref{opt:max-min-opt} through $({\rm LP}_{i})$. The optimal value achieves maximum for $i=2$ that $\gamma^{\star}_{2} = 300$, which implies a robust detection during the transient behavior as Corollary~\ref{cor:convex_rel}. For the given $\bar{N}$, the multivariate attack coordinates $\alpha = [2.8 \ 1 \ -2.3]^{\top}$ are obtained by solving the inner minimization of \eqref{opt:max-min-opt}. Next, we look into the steady-state behavior of the filter with the above sets $\mathcal{N}$ and  $\mathcal{A}$. For this, following Theorem~\ref{the:max-min-reform-nz-ne} we solve \eqref{opt:max-min-opt-nz} and \eqref{opt:min-max-opt-nz} through the programs~\eqref{opt:max-min-opt-ref-nz-lp} and \eqref{opt:min-max-opt-ref-nz-lp}. It turns out that the derived optimal values satisfy the equality $\varphi^\star = \mu^\star = 0$, indicating that the optimal multivariate attack with~$\alpha^\star$, the optimizer of the program~\eqref{opt:min-max-opt-ref-nz-lp} and an optimal solution to \eqref{opt:min-max-opt-nz}, is a stealthy attack in the long-term horizon. We highlight that, thanks to the fact that the optimal values of the programs~\eqref{opt:max-min-opt-ref-nz-lp} \eqref{opt:min-max-opt-ref-nz-lp} form a Nash equilibrium, even with the exact information of the stealthy attack coefficients~$\alpha^\star$, we still cannot decouple the long-term behavior of the residual from the natural disturbances; see Remark~\ref{rem:nash_equilibrium}.

In the first simulation, we begin with a general scenario where the multivariate attack is not carefully coordinated, i.e., basic attack. Thus as shown in Figure~\ref{subfig11:ad}, only 4 of 5 vulnerable measurements are compromised that $f_{tie_{1,2}} = 0.38 p.u.$, $f_{tie_{1}} = 0.53 p.u.$, $f_{tie_{2,3}} = -0.23 p.u.$ and  $f_{tie_{2}} = -0.23 p.u.$. Note that since the injected data on $\Delta P_{tie_{1,2}}$ and $\Delta P_{tie_{1}}$ are inconsistent, the static detector is also expected to be triggered. %
To test the detectors in a more realistic setup, we also consider the presence of process and measurements noises. %
The process noise term added to the state equation of Area 1 is zero-mean Gaussian noises with the covariance matrix $R_{X_{1}} = 0.03 \times \mbox{diag} (\left[ 1 \ 1 \ 0.03 \ 1 \ 1 \ 1 \ 1 \right]^\top)$, i.e., the covariance of the noise to the frequency is 0.009 and the covariance of other states' noise is 0.03 \cite{Ameli2018a}. Similarly, the measurement noise term added to the measurements of Area 1 is with the covariance matrix $R_{Y_{1}} = 0.03 \times \mbox{diag} (\left[ 1 \ 1 \ 1 \ 0.03 \ 1 \ 1 \ 1 \ 1 \ 1 \right]^\top)$, i.e., the covariance of the frequency measurement is 0.009 and the covariance of other measurements' noise is~0.03 \cite{Ameli2018a}. Note the residue $r_{S}$ of BDD in \eqref{eq:static_r} becomes $r_{S}[k] = (I - C({C}^\top R_{Y}^{-1} {C})^{-1}{C}^\top R_{Y}^{-1})Y[k]$ under the noisy system. The attacks are launched at $k_{\min} = {30}\ \mathrm{sec}$. In Figure~\ref{subfig14:srn} and Figure~\ref{subfig15:rrn}, results of the static detector in \eqref{eq:static_r} and the proposed dynamic detector (diagnosis filter) are presented. %In Figure~\ref{fig:resfig_(un)stealth_ideal} of Appendix B ~\ref{subsec:sim_res_ideal}, we also present the simulation results under ideal scenarios where there is no noise in the process and measurements. 
Both detectors have succeeded to generate a diagnostic signal when attacks occurred, and the diagnosis filter residual $r_{D}$ is significantly decoupled from stochastic load disturbances, and keeps sensitive to the multivariate attacks for a successful detection under noisy system settings.

In the second simulation, to challenge the detectors, now the multivariate attacks have been launched on all the 5 vulnerable measurements and the derived attack coefficient $\alpha$ from the optimization results has been used for a more intelligent adversary. Thus in Figure~\ref{subfig21:ad}, the corruptions become $f_{tie_{1,2}} = 0.38 p.u.$, $f_{tie_{1,3}} = 0.15 p.u.$, $f_{tie_{1}} = 0.53 p.u.$, $f_{tie_{2,3}} = -0.23 p.u.$ and  $f_{tie_{2}} = -0.23 p.u.$. This corresponds to the worst case for the diagnosis filter that the adversary is given the knowledge of the residual generator's parameter $\bar{N}$ that it tries to minimize the payoff function over $\mathcal{A}$. Besides, the noisy system settings have been considered. Figure~\ref{subfig24:srn} and Figure~\ref{subfig25:rrn} demonstrate all the simulation results. In Figure~\ref{subfig24:srn}, the static detector becomes totally blind to the occurrence of such an intelligent attack. However, as we can see in Figure~\ref{subfig25:rrn}, even in the worst case, the diagnosis filter works perfectly well under the noisy system, generate a residual ``alert''  for the presence of multivariate attacks. We can also see that the residual output becomes close to zero value again after a successful detection during the transient behavior in Figure~\ref{subfig25:rrn}, which is consistent to the aforementioned result $\varphi^\star = \mu^\star = 0$ and Remark~\ref{rem:nash_equilibrium}. These simulations also prove the effectiveness and robustness of the proposed diagnosis filter design. 

\subsection{Further discussions}% on filter sensitivity, other detectors and attacks}
In this section we elaborate several practical aspects of the proposed filter in the preceding section. 

\subsubsection{Diagnosis sensitivity to filter poles}
While the denominator of the filter~$a(q)$ in \eqref{eq:residual_gen} is chosen rather arbitrarily, up to a stability condition, the poles however has a significant impact on the residual sensitivity. As a general rule, the smaller the poles, the faster the residual responds, and the more sensitive the residual responds to model imprecision and noises. Simulation results in Figure~\ref{fig:resfig_diffp_unstealth} in Appendix~\ref{subsec:add_sim_res} numerically illustrate this relation when the filter poles vary. 

\subsubsection{Other types of attacks}
In addition to a smart multivariate measurement attacks, the main focus of this study, there are several other types of attacks that we briefly discuss in the following:
\begin{itemize}
	\item {\em Denial-of-service (DoS) attack}: A type of availability attack where the attacker aims to prevent some specific data from being delivered to the respective destinations. 
	\item {\em Replay attack}: A two-stage attack where the adversary gathers a sequence of data packets at stage 1, and then replays the recorded data afterwards at stage 2. 
\end{itemize}

From a detection point of view, DoS attacks are trivially detectable without any sophisticated mechanisms as the absence of data is not stealthy. In the typical DoS attack modeling, the missing data is typically replaced with the last received ones \cite{Schenato2009}. In such a mechanism, the DoS can be treated as an ``injection'' attack. We investigate the performance of our filter in the presence of this class of attacks in Figure~\ref{fig:resfig_dos} in Appendix~\ref{subsec:add_sim_res}. Numerical results confirm that the proposed filter can successfully detect the DoS attacks. In regard with the replay attack, the articles~\cite{Mo2009,Hoehn2016} offer sufficient conditions under which plausible attacks may remain stealthy irrespective of the detection mechanism providing that the attacker has access all the necessary data channels and excite attack of stage 2 at a suitable time. 

\subsubsection{Observer-based diagnosis filters}
%For a diagnosis tool design, the main attention is paid to,
%
%\begin{itemize}
%\item the on-line implementation requirements;
%\item which approach can be readily used to achieve the desired performance (e.g., robust detection of attacks). %perfect decoupling of disturbances, robust detection of all admissible attacks, etc.)
%\end{itemize}
%

%
%\begin{table}[t]
%	\caption{Comparison of different residual generator methods.}
%	\centering
%	\begin{tabular}{|p{0.1\textwidth}|p{0.06\textwidth}|p{0.06\textwidth}|p{0.06\textwidth}|p{0.06\textwidth}|}
%		\hline
%		Type                & \makecell*[tc]{State \\ observer- \\based} & \multicolumn{2}{c|}{\makecell*[tc]{Output \\ observer- \\based}} & \makecell*[tc]{Proposed \\ diagnosis \\filter} \\ \hline\hline
%		Order               & $s=n_X$              &     $s \leq n_X$     &    $s > n_x$        &   $d_{N} \ll n_X$         \\ \hline
%		Design param.       & $M, V$             &  $G, M, V, W$   & $G, M, V, W$  &  \makecell*[tc]{$N(\cdot)$, $a(\cdot)$}  \\ \hline
%		%Design freedom      &                      &                      &                     &                           \\ \hline
%		%Solution form       & LTI                  & LTI      &    LTI    &    LP                     \\ \hline
%		%Implement.          & recursive            & recursive            &  recursive          &   non-recursive                        \\ \hline
%		Robustness          & N/A                  & N/A                  &  N/A                &        Yes                 \\ \hline
%	\end{tabular}
%	\label{tab:residual_gens}
%\end{table}

Another major technique for anomaly detection builds on observer-based techniques. In this view, the estimate of the system states, or in more general setting {\em output observer}, is a reference to alert the abnormality ~\cite{Ge1988}. We close this section by a brief summary of the differences between these approaches and the one proposed in this study. 
%Before a comparison of these techniques with our diagnosis filter, we highlight their interconnections first: 
%In fact, these methods together with our diagnosis filter share a similar residual generator form of \eqref{eq:residual_gen} where the input is $y[\cdot]$ and the output is a residual signal $r[\cdot]$. 
%Differently, our diagnosis filter has a more straightforward form of linear transfer operations \eqref{eq:residual_gen}. 
%They all aim to decouple from the unknown signals ($x[\cdot]$) while keep sensitive to anomalies $f[\cdot]$. %they also share the same conditions of perfect decoupling (e.g., a check condition via Rosenbrock system matrix). 
% and the conditions for perfect decoupling are the same \cite[Lemma~4.2]{Esfahani2016}. 
%To this end, 
%Table~\ref{tab:residual_gens} lists several properties of the existing diagnosis filters in~\cite{Ding2008} along with the one proposed in this study. Here $s$ is the order of an observer; $G$, $M$, $V$, $W$ are the design matrices of the observer' state-space model. %residual generators.
%Solution form of LTI (linear time-invariant) indicates a needed knowledge of linear control systems, as opposed to LP (linear program) in our method. 
%To conclude, %a few remarks can be made,
%
\begin{itemize}
	\item The observer-based approaches typically yield diagnosis filters with higher dynamical system degrees than the approach proposed in this study. A low-order diagnosis filter is often more desired due to practical aspects of online implementation particularly for large-scale power systems. 
	
	\item Observer-based diagnosis filters usually rely on a precondition of system observability. An extended version of such filters relaxes this condition to the so-called Luenberger-type conditions~\cite{Andrieu2006}. Our diagnosis filter, however, requires a weaker condition reflected through the feasibility condition of the resulting optimization programs, e.g., when the program~\eqref{eq:robust-poly_bar} in Lemma~\ref{lem:robust_scheme_reform} is feasible.
	
	\item Thanks to the optimization-based framework, unlike the observer-based approaches, we have a systematic approach to incorporate a multivariate attack scenario into the framework. 
\end{itemize}
%
%
%

%=============================================================================== 
\section{Conclusion} 
\label{sec:conclusion}
%=============================================================================== 

In this article, we investigated the problem of anomaly detection in the power system cyber security with a particular focus on exploiting the dynamics information where tempering multiple measurements data may be possible. Our study showed that a dynamical perspective to the detection task indeed offers powerful diagnosis tools to encounter attack scenarios that may remain stealthy from a static point of view. The effectiveness of this result was validated by simulations in the IEEE 39-bus system. Future research directions that we envision include an extension to nonlinear systems, as well as a setting exposed to the ``dynamic'' (time-variant) attacks in Remark~\ref{rem:time-variant_att}, as opposed to the linear models and stationary attack scenarios studied in this article.

%=============================================================================== 
\setcounter{section}{1}

\section*{Appendix I: Technical Proofs}

\setcounter{subsection}{0}

\subsection{Proof of Theorem~\ref{the:max-min-reform}} \label{subsec:app_the_proof1}
Let us recall that $\bar{N} V(\alpha) = \left[\begin{matrix} N_{0}F\basis\alpha & N_{1}F\basis\alpha & \cdots & N_{d_{N}} F\basis\alpha \end{matrix}\right]$, and as such, the payoff function of the robust reformulation~\eqref{opt:max-min-opt} is $\mathcal{J}(\bar{N},\alpha)= \max_{i} |N_{i}F\basis\alpha|$ where $i \in \{ 0, \cdots, d_{N} \}$. By introducing an auxiliary variable $\beta$ in the simplex set $\mathcal{B}\Let\{\beta \in \R^{2d_{N}+2} \ | \ \beta \geq 0, \  \mathbf{1}^{\top}\beta = 1\}$, one can rewrite $\mathcal{J}$ as 
$$\mathcal{J}(\bar{N},\alpha) = \max\limits_{\beta \in \mathcal{B}} \ \sum_{i=0}^{d_{N}}(\beta_{2i}-\beta_{2i+1})N_{i}F\basis.$$ 
In this light, the original robust strategy~\eqref{opt:max-min-opt} can be equivalently described via 
\begin{equation}\label{dual}
\max\limits_{ \bar{N} \in \mathcal{N} } \, \min\limits_{\alpha \in \mathcal{A} } \, \max\limits_{\beta \in \mathcal{B}} \left\{ \sum_{i=0}^{d_{N}}(\beta_{2i}-\beta_{2i+1})N_{i}F\basis\alpha \right\}. \nonumber
\end{equation}

Note that given a fixed $\bar{N}$ the inner minimax optimization is indeed a bilinear objective in the decision variables and the respective feasible sets $\mathcal A$ and $\mathcal B$ are convex. Since one of the sets, $\mathcal B$, is also compact, then the zero-duality gap holds. Therefore, interchanging the optimization over $\alpha \in \mathcal A$ and $\beta \in \mathcal B$ yields
\begin{equation} \label{opt:max-min-beta}
\gamma^\star = \max\limits_{ \bar{N} \in \mathcal{N}, \ \beta \in \mathcal{B} } \ \left\{\min\limits_{\alpha \in \mathcal{A} } \ \sum_{i=0}^{d_{N}}(\beta_{2i}-\beta_{2i+1})N_{i}F\basis\alpha \right\}. 
\end{equation}
The inner minimization of \eqref{opt:max-min-beta} is a (feasible) linear program. We can use the duality again. To this end, let us assume that the decision variables~$\bar{N}$ and $\beta$ are fixed and consider the Lagrangian function
\begin{align}
\mathcal L(\alpha;\lambda) = b^\top\lambda  + \Big(\sum_{i=0}^{d_{N}}(\beta_{2i}-\beta_{2i+1})N_{i}F\basis - \lambda^{\top}A\Big)\alpha, \nonumber
\end{align}
where optimizing over an unconstrained variable~$\alpha$ becomes 
\begin{align*}
\min_{\alpha}\mathcal L(\alpha;\lambda) = \left\{
\begin{array}{ll}
\ b^\top\lambda & \mbox{if} ~ \left\{\begin{array}{l}\sum\limits_{i=0}^{d_{N}}(\beta_{2i}-\beta_{2i+1})N_{i}F\basis = \lambda^{\top}A \\ \lambda \geq 0 \end{array} \right. \vspace{1mm}\\
-\infty & \mbox{otherwise,}
\end{array}
\right.\nonumber
\end{align*}
Using the above characterization as the most inner optimization program in \eqref{dual} leads to
\begin{align}\label{opt:innermin_dual}
&\begin{array}{ccl} 
& \max\limits_{\lambda} & b^{\top}\lambda \\
& \mbox{s.t.} & \sum\limits_{i=0}^{d_{N}}(\beta_{2i}-\beta_{2i+1})N_{i}F\basis = \lambda^{\top}A,\\
& & \lambda \geq 0.
\end{array}
\end{align}

It then suffices to combine maximizing over the auxiliary variable~$\lambda$ together with the variables $\bar N$ and $\beta$ to arrive at the main result in \eqref{opt:max-min-ref}.

\subsection{Proof of Theorem~\ref{the:max-min-reform-nz-ne}} \label{subsec:app_the_proof_the2}
We first prove the convex reformulation. For a given $\bar{N} \in \mathcal N$, the inner minimization of \eqref{opt:max-min-opt-nz} can be translated as 
\begin{align}\label{opt:max-min-nz-innermin}
&\begin{array}{ccl} 
& \min\limits_{\alpha \in \mathcal{A}, \, r} &  r  \\
& \mbox{s.t.} & \bar{N}\bar{F}\alpha - r \leq 0, \\
& & -\bar{N}\bar{F}\alpha - r \leq 0. \nonumber
\end{array}
\end{align}	
The Lagrangian of the inner minimization %over the variable $\alpha$ 
reads as 
\begin{equation}
\mathcal L(\alpha, \, r; \, \beta, \, \lambda) = b^\top \lambda  + \big((\beta_{0}-\beta_{1})\bar{N}\bar{F} - \lambda^{\top}A\big)\alpha + (1-\beta_{0}-\beta_{1})r. \nonumber \\
\end{equation}
%
%where $\beta_{0}$, $\beta_{1}$, $\lambda$ are dual variables. 
Optimizing over the variables~$\alpha, \ r$ yields %the dual function,
\begin{align*}
\min_{\alpha, \ r} \mathcal L(\alpha, \, r; \, \beta, \, \lambda) = \left\{
\begin{array}{ll}
\ b^\top\lambda & \mbox{if} ~ 
\left\{\begin{array}{l} (\beta_{0}-\beta_{1})\bar{N}\bar{F} = \lambda^{\top}A \\
\beta_{0}+\beta_{1} \leq 1 \\
\beta_0 \geq 0, \ \beta_1 \geq 0, \ \lambda \geq 0 \end{array} \right. \vspace{1mm}\\
-\infty &  \mbox{otherwise.}
\end{array}
\right.\nonumber
\end{align*}
Then, combining maximization over the auxiliary variables $\lambda$, $\beta_0$, $\beta_1$ together with the variable $\bar{N}$ arrives at the optimization program,
\begin{align}\label{opt:max-min-opt-ref-nz}
&\begin{array}{ccl}
\mu^\star = & \max\limits_{\bar{N}, \, \beta_0, \, \beta_1, \, \lambda} &  b^{\top} \lambda \\
& \mbox{s.t.} & (\beta_{0}-\beta_{1})\bar{N}\bar{F}= \lambda^{\top}A, \\
& & \beta_{0}+\beta_{1} \leq 1, \ \beta_0 \geq 0, \ \beta_1 \geq 0, \\
& & \bar{N} \in \mathcal{N}, \ \lambda \in \R^{n_b}, \ \lambda \geq 0.  
\end{array}
\end{align}
Note that the actual program~\eqref{opt:max-min-opt-ref-nz-lp} is a restriction of \eqref{opt:max-min-opt-ref-nz} where the variables~$\beta_0$ and $\beta_1$ are restricted to $\beta_0 = 1$ and $\beta_1= 0$. Next, we show that this restriction is indeed without loss of generality. To this end, suppose the tuple ($\beta_0^\star$, $\beta_1^\star$, $\bar{N}^\star$, $\lambda^\star$) is an optimal solution to the program~\eqref{opt:max-min-opt-ref-nz}. %It is straightforward to see that if an $\bar{N}^\ast$ is a optimal solution to \eqref{opt:max-min-opt-ref-nz}, then so is $-\bar{N}^\ast$ thanks to readjusting the optimal variables $v_1^\ast, v_2^\ast$; recall that $\mathcal{N}$ is a symmetric set. 
Note that the optimal variables $\beta_{0}^\star$ and $\beta_{1}^\star$ may satisfy one of the following three properties:

\begin{itemize}
	\item [(i)] $\beta_0^\star = \beta_1^\star$: In this case, $\lambda^\star = 0$, and therefore the optimal value $\mu^\star=0$. This optimal solution can be trivially achieved in the program~ \eqref{opt:max-min-opt-ref-nz-lp} by setting $\bar N = 0$.
	\item [(ii)] $\beta_0^\star > \beta_1^\star$: Observe that the tuple~$\big(\beta_{0}^{'}=1$,  $\beta_{1}^{'}=0$, $\bar{N}^{'}=\bar{N}^\star$, $\lambda^{'}={\lambda^\star}/{(\beta_0^\star - \beta_1^\star)}\big)$ is a feasible solution with the objective value~${b^\top \lambda^\star}/{(\beta_0^\star - \beta_1^\star)}$. Since $b^\top \lambda^\star \ge 0$ by optimality assumption and $\beta_0^\star - \beta_1^\star \in (0,1]$, then this feasible solution has a possibly higher optimal value, and therefore $\beta_0^\star - \beta_1^\star = 1$. That is, $\beta_0^\star = 1$ and $\beta_1^\star = 0$. 
	\item [(iii)] $\beta_0^\star < \beta_1^\star$: Following similar steps as the previous case together with the symmetric property of the feasible set~$\mathcal N$, one can show that the optimal value of the program~\eqref{opt:max-min-opt-ref-nz} also coincides with the restricted version in \eqref{opt:max-min-opt-ref-nz-lp}. 
\end{itemize}

This concludes the proof of the convex reformulation from \eqref{opt:max-min-opt-nz} to \eqref{opt:max-min-opt-ref-nz-lp}. In regard with the minimax problem \eqref{opt:min-max-opt-nz}, let us recall the symmetric property of the feasible set $\mathcal{N}$ in the variable $\bar{N}$. With a fixed $\alpha$, the inner maximization can be directly formed as $\max_{\bar{N} \in \mathcal{N}} \bar{N} \bar{F} \alpha $ whose Lagrangian becomes 
\begin{equation}
\mathcal L(\bar{N}; v, w) = - (\mathbf{1}^{\top}v_{1} + \mathbf{1}^{\top}v_{2})  + \big(w^\top \bar{H}^\top + v_{1}^\top - v_{2}^\top - (\bar{F}\alpha)^\top\big)\bar{N}^{\top}, \nonumber \\
\end{equation}
%
%where $v_{1}$, $v_{2}$, $w$ are dual variables. 
Optimizing over the variable~$\bar{N}$ leads to
\begin{align*}
\min_{\bar{N}} \mathcal L(\bar{N}; v, w) = \left\{
\begin{array}{ll}
-\mathbf{1}^{\top}v_{1} - \mathbf{1}^{\top}v_{2} & \mbox{if}~
\left\{\begin{array}{l}\bar{H}w + v_{1} - v_{2} = \bar{F}\alpha \\
v_1 \geq 0, \  v_2 \geq 0 \end{array} \right. \vspace{1mm}\\
-\infty & \mbox{otherwise.}
\end{array}
\right.\nonumber
\end{align*}
Thus, combining minimization over the auxiliary variables $v_{1}, \ v_{2}, \ w$ together with the variable $\alpha$, the minimax optimization \eqref{opt:min-max-opt-nz} can be reformulated as the linear program \eqref{opt:min-max-opt-ref-nz-lp}.

Finally, we show that the solution to programs~\eqref{opt:thm:lp} indeed forms a Nash equilibrium between the programs~\eqref{opt:max-min-opt-nz} and \eqref{opt:min-max-opt-nz}. Thus far, we have reformulated maximin and minimax problems as linear programs~\eqref{opt:thm:lp}. The idea is to show that these programs have the same optimal values. In fact, we show that the programs are dual of each other, and that the strong duality holds when both programs are feasible. To this end, we resort to the duality of \eqref{opt:max-min-opt-ref-nz-lp} with the Lagrangian 
\begin{align}
\mathcal L(\bar{N}, \lambda; \alpha, v, w) & = \big(w^\top \bar{H}^\top + v_{1}^\top - v_{2}^\top - (\bar{F}\alpha)^\top\big)\bar{N}^{\top}+(\alpha^\top A^{\top} - b^{\top}) \lambda  -(\mathbf{1}^{\top}v_{1} + \mathbf{1}^{\top}v_{2}) . \nonumber
\end{align}
%
%The Lagrangian dual function is 
Optimizing over the variables~$\bar{N}$, $\lambda$ yields
\begin{align*}
\min_{\bar{N},\lambda} \mathcal L(\bar{N}, \lambda; \alpha, v, w)  = \left\{
\begin{array}{ll}
-\mathbf{1}^{\top}v_{1} - \mathbf{1}^{\top}v_{2} & \mbox{if} 
\left\{\begin{array}{l} \bar{H}w + v_{1} - v_{2} = \bar{F}\alpha \\
A\alpha \geq b \\
v_1 \geq 0, \ v_2 \geq 0 \end{array} \right. \vspace{1mm}\\
-\infty & \mbox{otherwise.}
\end{array}
\right.\nonumber
\end{align*}
It is not difficult to see that the above program coincides with the program~\eqref{opt:min-max-opt-ref-nz-lp}; this concludes the proof.

\setcounter{section}{2}
\section*{Appendix II: System Parameters \& Added Simulation Results}

\setcounter{subsection}{0}

\subsection{Dynamic Feedback Controller Modeling}
\label{subsec:app_rem_dyn_ctrl}

Consider a dynamical system (e.g., the electrical power system studied in Section~\ref{sec:powersys_model}). Suppose the control signal is implemented as a {\em dynamic} feedback controller described by the discrete-time dynamics 
\begin{equation}\label{eq:dyn_ctrl}
\left\{
\begin{aligned}
& X_{c}[k+1] = A_{c}X_{c}[k] + B_{c}Y[k], \nonumber\\
& u[k] =   C_{c}X_{c}[k] + D_{c}Y[k], \nonumber
\end{aligned}
\right.
\end{equation}
where the input is the dynamical system measurements~$Y[\cdot]$, the output the control signal~$u[\cdot]$, and the internal state of the controller is denoted by $X_{c} \in \R^{n_{c}}$. When an attack occurs on the measurements, it affects the dynamics of the controller and consequently the involved physical system. To study the control dynamics together with the original dynamical system, one can augment the states of the system~\eqref{eq:dis_sys} together with the controller's as $\hat{X}\Let[X^{\top} \ X_{c}^{\top}]^{\top}$. Assuming that the control signal can also be measured, one can also introduce an augmented measurement signals as~$\hat{Y}=[Y^{\top} \ u^{\top}]^{\top}$. Following this procedure, the dynamics of the closed-loop system is described by
\begin{equation}\label{eq:full_sys}
\left\{
\begin{aligned}
& \hat{X}[k+1] = \hat{A}_{cl} \hat{X}[k]  + \hat{B}_{d} d[k] + \hat{B}_{f} f[k], \\
& \hat{Y}[k] =  \hat{C}\hat{X}[k]  + \hat{D}_f f[k].
\end{aligned}
\right.
\end{equation}
where the involved matrices are defined as
\begin{equation}\label{eq:sys_matrix_def}
\begin{aligned}
&\hat{A}_{cl} \Let \left[\begin{matrix} A_{x}+B_u D_c C & B_u C_c \\ B_c C & A_{c} \end{matrix}\right], \quad \hat{B}_{d} \Let \left[\begin{matrix} B_{d} \\ 0 \end{matrix}\right], \quad \hat{B}_{f} \Let \left[\begin{matrix} B_u D_c D_f \\ B_c D_f \end{matrix}\right], \nonumber\\
&\quad \quad \quad \quad \quad \hat{C} \Let \left[\begin{matrix} C & 0 \\ D_c C & C_{c} \end{matrix}\right], \quad  \hat{D}_f \Let \left[\begin{matrix} D_f \\ D_c D_f \end{matrix}\right] \,. \nonumber
\end{aligned}
\end{equation}
In this view, the augmented system~\eqref{eq:full_sys} shares the same structure as \eqref{eq:dis_cls} studied in the main part of the article for the case of static feedback controller.

\begin{figure*}[t!p]
	\centering
	\begin{subfigure}[t]{0.49\textwidth}
		\centering
		\includegraphics[scale=0.52]{fig11ad_unstealth_p08_0403.pdf}
		\caption{Load disturbance and basic attack}\label{subfig11:add_ideal}
	\end{subfigure}
	~
	\begin{subfigure}[t]{0.49\textwidth}
		\centering
		\includegraphics[scale=0.52]{fig21ad_stealth_p08_0403.pdf}
		\caption{Load disturbance and stealthy attack}\label{subfig21:add_ideal}
	\end{subfigure}
	\\
	\begin{subfigure}[t]{0.49\textwidth}
		\centering
		\includegraphics[scale=0.52]{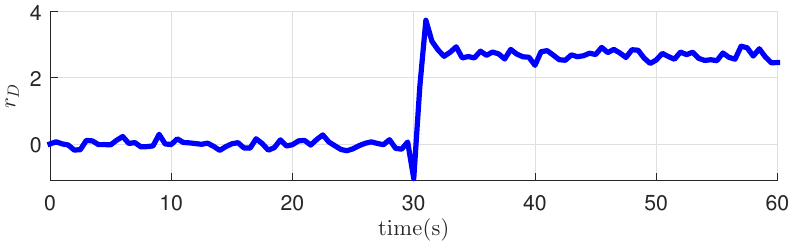}
		\caption{Residual signal $r_{D}$ with pole $p=0.1$}\label{subfig31:r_p01}
	\end{subfigure}
	~
	\begin{subfigure}[t]{0.49\textwidth}
		\centering
		\includegraphics[scale=0.52]{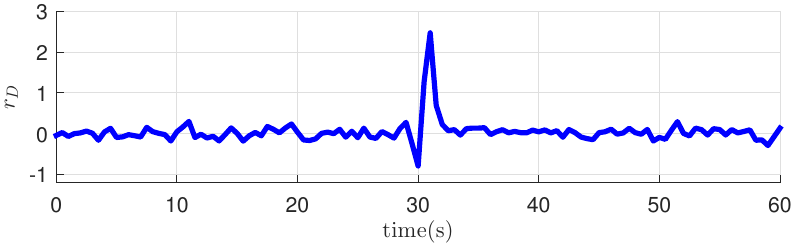}
		\caption{Residual signal $r_{D}$ with pole $p=0.1$}\label{subfig41:r_p01}
	\end{subfigure}
	\\
	\begin{subfigure}[t]{0.49\textwidth}
		\centering
		\includegraphics[scale=0.52]{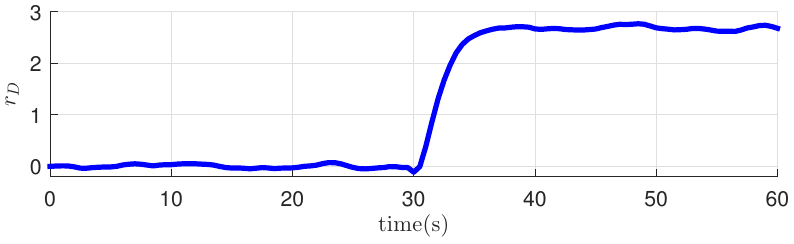}
		\caption{Residual signal $r_{D}$ with pole $p=0.6$}\label{subfig31:r_p06}
	\end{subfigure}
	~
	\begin{subfigure}[t]{0.49\textwidth}
		\centering
		\includegraphics[scale=0.52]{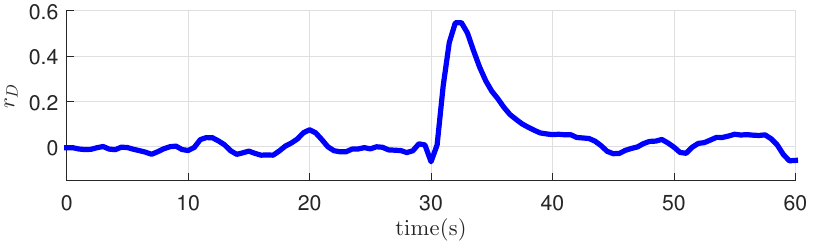}
		\caption{Residual signal $r_{D}$ with pole $p=0.6$}\label{subfig41:r_p06}
	\end{subfigure}
	\\ 
	\begin{subfigure}[t]{0.49\textwidth}
		\centering
		\includegraphics[scale=0.52]{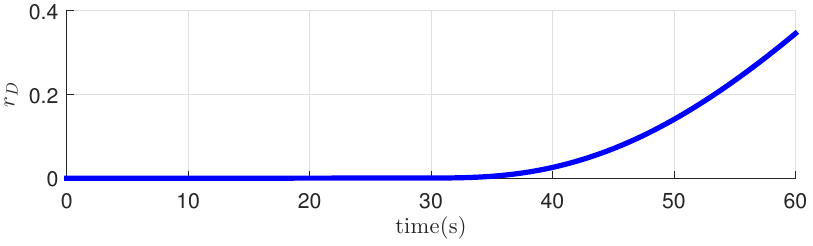}
		\caption{Residual signal $r_{D}$ with pole $p=0.98$}\label{subfig31:r_p098}
	\end{subfigure}
	~
	\begin{subfigure}[t]{0.49\textwidth}
		\centering
		\includegraphics[scale=0.52]{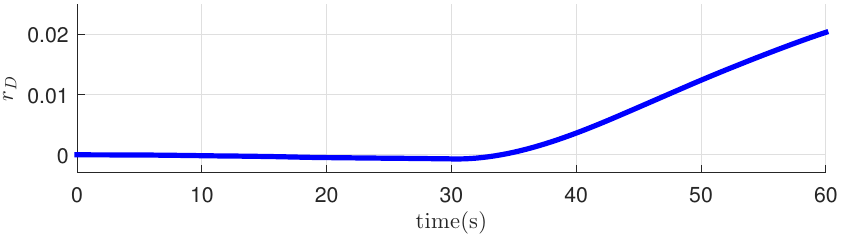}
		\caption{Residual signal $r_{D}$ with pole $p=0.98$}\label{subfig41:r_p098}
	\end{subfigure}
	\\
	\caption{Results of dynamic detector (diagnosis filter) with different {\em poles} ($p=0.1, \ 0.6, \ 0.98$) under basic and stealthy attacks.} \label{fig:resfig_diffp_unstealth}
\end{figure*}

\subsection{AGC Parameters of the three-area 39-bus system}\label{subsec:agc_39_sys_matrices}
In this subsection we provide the involved matrices and parameters of the three-area 39 system. We take the model description of Area 1 in the three-area system in Figure~\ref{fig:39bus} of Section~\ref{sec:powersys_model} as an instance,
\begin{equation} \label{eq:spx_B1d}
B_{1,d} = \left[\begin{matrix} 0 & 0 & -\frac{1}{2H_{1}} & 0 & 0 & 0 \end{matrix}\right]^{\top}, \nonumber
\end{equation}
\begin{gather}\label{eq:spx_A11}
A_{11} = \left[\begin{matrix} 0 & 0 & T_{12} & 0 & 0 & 0 \\ 0 & 0 & T_{13} & 0 & 0 & 0 \\ -\frac{1}{2H_{1}} & -\frac{1}{2H_{1}} & -\frac{D_{1}}{2H_{1}} & \frac{1}{2H_{1}} & \frac{1}{2H_{1}} & 0 \\ 0 & 0 & -\frac{1}{T_{ch_{1,1}}S_{1,1}} & -\frac{1}{T_{ch_{1,1}}} & 0 & \frac{\phi_{1,1}}{T_{ch_{1,1}}} \\ 0 &  0 & -\frac{1}{T_{ch_{1,2}}S_{1,2}} & 0 & -\frac{1}{T_{ch_{1,2}}} & \frac{\phi_{1,2}}{T_{ch_{1,2}}} \\ -K_{I_{1}} & -K_{I_{1}} & -K_{I_{1}}B_{1} & 0 & 0 & 0 \end{matrix}\right] \, . \nonumber
\end{gather}
%\end{equation}
%

As we have assumed a measurement model with high redundancy, the matrix $C_{i}$ for Area 1 becomes
\begin{equation} \label{eq:spy_c1}
C_{1} = \left[\begin{matrix} 1 & 0 & 0 & 0 & 0 & 0 \\ 0 & 1 & 0 & 0 & 0 & 0 \\ 0 & 0 & 1 & 0 & 0 & 0 \\ 0 & 0 & 0 & 1 & 0 & 0 \\ 0 & 0 & 0 & 0 & 1 & 0 \\ 0 & 0 & 0 & 0 & 0 & 1 \\ 1 & 1 & 0 & 0 & 0 & 0 \\ 0 & 0 & 0 & 1 & 1 & 0\end{matrix}\right]^{\top}. \nonumber
\end{equation}

In Area 1, the vulnerable measurements to cyber attacks are the ones of tie-line power flows $\Delta P_{tie_{1,2}}$, $\Delta P_{tie_{1,3}}$ and $\Delta P_{tie_{1}}$. Thus the AGC signal $\Delta P_{agc_{1}}$ would be corrupted into
\begin{equation}\label{eq:agc_f}
\Delta \dot{P}_{agc_{1}} =  - k_{1}(B_{1} \Delta \omega_{1} +  \Delta P_{tie_{1,2}} + f_{tie_{1,2}} + \Delta P_{tie_{1,3}} + f_{tie_{1,3}}) \, . \nonumber
\end{equation}
Then the parameters regarding multivariate attacks are
\begin{equation} \label{eq:spy_f1}
f_{1} = \left[\begin{matrix} f_{tie_{1,2}} & f_{tie_{1,3}} & f_{tie_{1}}\end{matrix}\right]^{\top}, \nonumber
\end{equation}
\begin{equation} \label{eq:spy_D1f}
D_{1,f} = \left[\begin{matrix} 1 & 0 & 0 & 0 & 0 & 0 & 0 & 0 \\ 0 & 1 & 0 & 0 & 0 & 0 & 0 & 0 \\ 0 & 0 & 0 & 0 & 0 & 0 & 1 & 0 \end{matrix}\right]^{\top}, \quad B_{1,f} = \left[\begin{matrix} 0 & 0 & 0 & 0 & 0 & -k_{1} \\ 0 & 0 & 0 & 0 & 0 & -k_{1} \\ 0 & 0 & 0 & 0 & 0 & 0 \end{matrix}\right]^{\top}. \nonumber
\end{equation}

\subsection{Additional simulation results}\label{subsec:add_sim_res}
\begin{figure}[t!]
	\centering
	\begin{subfigure}{0.49\textwidth}
		\centering
		\includegraphics[scale=0.52]{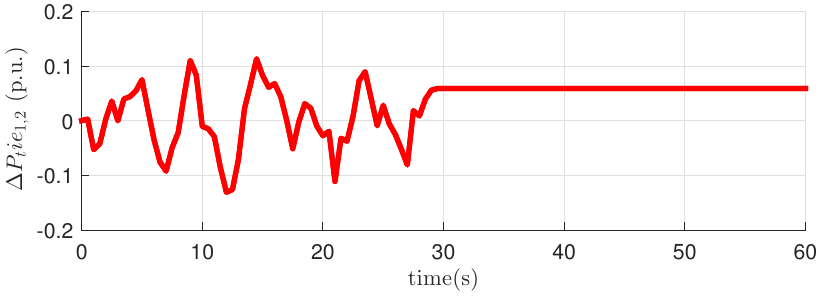}
		\caption{$\Delta P_{tie_{1,2}}$ under DoS attacks from $k_{dos}=30$}\label{subfig31:pdos}
	\end{subfigure} ~ 
	\begin{subfigure}{0.49\textwidth}
		\centering
		\includegraphics[scale=0.52]{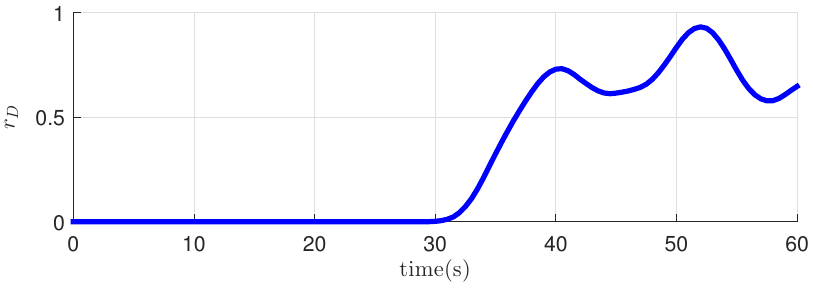}
		\caption{Residual signal $r_{D}$ under DoS attacks}\label{subfig32:rdos}
	\end{subfigure}
	\caption{Results of dynamic detector (diagnosis filter) under DoS attacks on $\Delta P_{tie_{1,2}}$ ($p=0.8$). }
	\label{fig:resfig_dos}
\end{figure}
In Figure~\ref{fig:resfig_diffp_unstealth} we present the simulation results of the residual signal $r_{D}$ from the proposed diagnosis filter under different {\em poles} ($p=0.1, \ 0.6, \ 0.98$, respectively). We also show the simulation results of the residual signal $r_{D}$ from the proposed diagnosis filter under DoS attacks in Figure~\ref{fig:resfig_dos}. 
%

%===============================================================================
	\bibliographystyle{siam}
	\bibliography{literature}

\begin{thebibliography}{10}

\bibitem{Ameli2018a}
{\sc A.~Ameli, A.~Hooshyar, E.~El-Saadany, and A.~Youssef}, {\em Attack
  detection and identification for automatic generation control systems}, IEEE
  Transactions on Power Systems,  (2018), p.~1.

\bibitem{Andrieu2006}
{\sc V.~Andrieu and L.~Praly}, {\em On the existence of a
  kazantzis--kravaris/luenberger observer}, {SIAM} Journal on Control and
  Optimization, 45 (2006), pp.~432--456.

\bibitem{Ashok2018}
{\sc A.~{Ashok}, M.~{Govindarasu}, and V.~{Ajjarapu}}, {\em Online detection of
  stealthy false data injection attacks in power system state estimation}, IEEE
  Transactions on Smart Grid, 9 (2018), pp.~1636--1646.

\bibitem{bevrani2008}
{\sc H.~Bevrani}, {\em Robust Power System Frequency Control}, Power
  Electronics and Power Systems, Springer, 2008.

\bibitem{Chen2018}
{\sc C.~Chen, K.~Zhang, K.~Yuan, L.~Zhu, and M.~Qian}, {\em Novel detection
  scheme design considering cyber attacks on load frequency control}, IEEE
  Transactions on Industrial Informatics, 14 (2018), pp.~1932--1941.

\bibitem{ChenAbu-Nimeh2011}
{\sc T.~M. Chen and S.~Abu-Nimeh}, {\em Lessons from stuxnet}, Computer, 44
  (2011), pp.~91--93.

\bibitem{NIST2018}
{\sc C.~Cybersecurity}, {\em Framework for improving critical infrastructure
  cybersecurity version 1.1}, tech. rep., National Institute of Standards and
  Technology, Apr. 2018.

\bibitem{Deng2018}
{\sc R.~Deng and H.~Liang}, {\em False data injection attacks with limited
  susceptance information and new countermeasures in smart grid}, IEEE
  Transactions on Industrial Informatics,  (2018), p.~1.

\bibitem{Ding2008}
{\sc S.~X. Ding}, {\em Model-based fault diagnosis techniques: design schemes,
  algorithms, and tools}, Springer Science \& Business Media, 2008.

\bibitem{Gao2017}
{\sc X.~Gao, X.~Liu, and J.~Han}, {\em Reduced order unknown input observer
  based distributed fault detection for multi-agent systems}, Journal of the
  Franklin Institute, 354 (2017), pp.~1464--1483.

\bibitem{Ge1988}
{\sc W.~Ge and C.-Z. FANG}, {\em Detection of faulty components via robust
  observation}, International Journal of Control, 47 (1988), pp.~581--599.

\bibitem{Giani2013}
{\sc A.~{Giani}, E.~{Bitar}, M.~{Garcia}, M.~{McQueen}, P.~{Khargonekar}, and
  K.~{Poolla}}, {\em Smart grid data integrity attacks}, IEEE Transactions on
  Smart Grid, 4 (2013), pp.~1244--1253.

\bibitem{Gorman2009}
{\sc S.~Gorman}, {\em Electricity grid in {US} penetrated by spies}, The Wall
  Street Journal, 8 (2009).

\bibitem{Hoehn2016}
{\sc A.~Hoehn and P.~Zhang}, {\em Detection of replay attacks in cyber-physical
  systems}, in American Control Conference, 2016, pp.~290--295.

\bibitem{Hug2012}
{\sc G.~Hug and J.~A. Giampapa}, {\em Vulnerability assessment of {AC} state
  estimation with respect to false data injection cyber-attacks}, IEEE
  Transactions on Smart Grid, 3 (2012), pp.~1362--1370.

\bibitem{Li2015}
{\sc S.~Li, Y.~Y{\i}lmaz, and X.~Wang}, {\em Quickest detection of false data
  injection attack in wide-area smart grids}, IEEE Transactions on Smart Grid,
  6 (2015), pp.~2725--2735.

\bibitem{Liang2017a}
{\sc G.~Liang, S.~R. Weller, J.~Zhao, F.~Luo, and Z.~Y. Dong}, {\em The 2015
  {Ukraine} blackout: Implications for false data injection attacks}, IEEE
  Transactions on Power Systems, 32 (2017), pp.~3317--3318.

\bibitem{Liang2016}
{\sc J.~Liang, L.~Sankar, and O.~Kosut}, {\em Vulnerability analysis and
  consequences of false data injection attack on power system state
  estimation}, IEEE Transactions on Power Systems, 31 (2016), pp.~3864--3872.

\bibitem{Liu2014a}
{\sc L.~Liu, M.~Esmalifalak, Q.~Ding, V.~A. Emesih, and Z.~Han}, {\em Detecting
  false data injection attacks on power grid by sparse optimization}, IEEE
  Transactions on Smart Grid, 5 (2014), pp.~612--621.

\bibitem{LiuNingReiter2009}
{\sc Y.~Liu, P.~Ning, and M.~K. Reiter}, {\em False data injection attacks
  against state estimation in electric power grids}, in 16th ACM Conference on
  Computer and Communication Security, New York, 2009, pp.~21--32.

\bibitem{Massoumnia1989}
{\sc M.~A. Massoumnia, G.~C. Verghese, and A.~S. Willsky}, {\em Failure
  detection and identification}, IEEE Transactions on Automatic Control, 34
  (1989), pp.~316--321.

\bibitem{Mo2009}
{\sc Y.~Mo and B.~Sinopoli}, {\em Secure control against replay attacks}, in
  47th Annual Allerton Conference on Communication, Control, and Computing,
  2009, pp.~911--918.

\bibitem{Esfahani2016}
{\sc P.~{Mohajerin Esfahani} and J.~Lygeros}, {\em A tractable fault detection
  and isolation approach for nonlinear systems with probabilistic performance},
  IEEE Transactions on Automatic Control, 61 (2016), pp.~633--647.

\bibitem{mohajerin2010cyber}
{\sc P.~{Mohajerin Esfahani}, M.~Vrakopoulou, K.~Margellos, J.~Lygeros, and
  G.~Andersson}, {\em Cyber attack in a two-area power system: Impact
  identification using reachability}, in American Control Conference, 2010,
  pp.~962--967.

\bibitem{Nyberg2006}
{\sc M.~Nyberg and E.~Frisk}, {\em Residual generation for fault diagnosis of
  systems described by linear differential-algebraic equations}, IEEE
  Transactions on Automatic Control, 51 (2006), pp.~1995--2000.

\bibitem{Ogata1995}
{\sc K.~Ogata}, {\em Discrete-time Control Systems (2Nd Ed.)}, Prentice-Hall,
  Inc., Upper Saddle River, NJ, USA, 1995.

\bibitem{Pan2018}
{\sc K.~Pan, A.~Teixeira, M.~Cvetkovic, and P.~Palensky}, {\em Cyber risk
  analysis of combined data attacks against power system state estimation},
  IEEE Transactions on Smart Grid,  (2018), p.~1.

\bibitem{Rakhshani2017}
{\sc E.~Rakhshani, D.~Remon, A.~M. Cantarellas, J.~M. Garcia, and
  P.~Rodriguez}, {\em Virtual synchronous power strategy for multiple hvdc
  interconnections of multi-area agc power systems}, IEEE Transactions on Power
  Systems, 32 (2017), pp.~1665--1677.

\bibitem{Sahabandu2019}
{\sc D.~{Sahabandu}, S.~{Moothedath}, L.~{Bushnell}, R.~{Poovendran},
  J.~{Aller}, W.~{Lee}, and A.~{Clark}}, {\em A game theoretic approach for
  dynamic information flow tracking with conditional branching}, in American
  Control Conference, 2019, pp.~2289--2296.

\bibitem{Sahabandu2018}
{\sc D.~{Sahabandu}, B.~{Xiao}, A.~{Clark}, S.~{Lee}, W.~{Lee}, and
  R.~{Poovendran}}, {\em Dift games: Dynamic information flow tracking games
  for advanced persistent threats}, in IEEE Conference on Decision and Control,
  Dec. 2018, pp.~1136--1143.

\bibitem{Schenato2009}
{\sc L.~Schenato}, {\em To zero or to hold control inputs with lossy links?},
  IEEE Transactions on Automatic Control, 54 (2009), pp.~1093--1099.

\bibitem{Shukla2019}
{\sc P.~{Shukla}, A.~{Chakrabortty}, and A.~{Duel-Hallen}}, {\em A
  cyber-security investment game for networked control systems}, in American
  Control Conference, 2019, pp.~2297--2302.

\bibitem{Teixeira2010}
{\sc A.~Teixeira, S.~Amin, H.~Sandberg, K.~H. Johansson, and S.~S. Sastry},
  {\em Cyber security analysis of state estimators in electric power systems},
  in IEEE Conference on Decision and Control, 2010.

\bibitem{Tiniou2013}
{\sc E.~E. Tiniou, P.~{Mohajerin Esfahani}, and J.~Lygeros}, {\em Fault
  detection with discrete-time measurements: An application for the cyber
  security of power networks}, in IEEE Conference on Decision and Control,
  2013.

\bibitem{Zhao2018}
{\sc J.~Zhao, L.~Mili, and M.~Wang}, {\em A generalized false data injection
  attacks against power system nonlinear state estimator and countermeasures},
  IEEE Transactions on Power Systems,  (2018), p.~1.

\end{thebibliography}
%===============================================================================
\end{document}